\documentclass[reqno]{amsart}
\usepackage{amsthm,amsmath,amssymb}
\usepackage{stmaryrd,mathrsfs}
\usepackage{esint}
\usepackage{tikz}
\usepackage{hyperref}
\usepackage{enumerate}

\allowdisplaybreaks

\newcommand{\ud}[0]{\,\mathrm{d}}

\newcommand{\abs}[1]{|#1|}

\newcommand{\Norm}[2]{\|#1\|_{#2}}

\newcommand{\pair}[2]{\langle #1,#2 \rangle}

\newcommand{\bddlin}[0]{\mathscr{L}}

\newcommand{\loc}[0]{\operatorname{loc}}

\newcommand{\sign}[0]{\operatorname{sgn}}

\newcommand{\R}{\mathbb{R}}
\newcommand{\C}{\mathbb{C}}
\newcommand{\N}{\mathbb{N}}
\newcommand{\Z}{\mathbb{Z}}
\newcommand{\T}{\mathbb{T}}

\newcommand{\eps}[0]{\varepsilon}

\newcommand{\dbar}{{\mkern3mu\mathchar'26\mkern-12mu d}}

\swapnumbers \numberwithin{equation}{section}

\theoremstyle{plain}
\newtheorem{theorem}[equation]{Theorem}
\newtheorem{proposition}[equation]{Proposition}
\newtheorem{corollary}[equation]{Corollary}
\newtheorem{lemma}[equation]{Lemma}

\theoremstyle{definition}

\theoremstyle{remark}
\newtheorem{remark}[equation]{Remark}
\newtheorem{example}[equation]{Example}

\makeatletter
\@namedef{subjclassname@2010}{%
  \textup{2010} Mathematics Subject Classification}
\makeatother

\begin{document}

\title[Rough spectral asymptotics for commutators]{Rough spectral asymptotics for commutators of general singular integrals}

\author{Tuomas Hyt\"onen}
\address{Aalto University,
Department of Mathematics and Systems Analysis,
P.O. Box 11100, FI-00076 Aalto, Finland}
\email{tuomas.hytonen@aalto.fi}

\thanks{The author was supported by the Research Council of Finland, projects 364208 and 371637.}





\begin{abstract}
Recently, Schatten class membership of commutators arising from Connes' quantised calculus has been characterised in a general framework, which covers multiple concrete situations of interest. In contrast to this, related results dealing with the asymptotic behaviour of the singular values of these commutators have been restricted to a relatively short list of specific examples only. In this work, we show that a version of the spectral asymptotic formula, involving comparable lower and upper limits in place of an exact limit, remains valid in great generality of singular integrals over metric measure spaces. A key proof ingredient of independent interest is a new asymptotic version of the Marcinkiewicz interpolation theorem. We also present a new approach to commutator upper bounds at the critical index; while not quite as general as more elaborate methods, it is general enough to reproduce the known upper bounds in Heisenberg and Carnot groups in a much simpler way.
\end{abstract}

\maketitle

\setcounter{tocdepth}{1}

\tableofcontents

\section{Introduction}

In the quantised calculus of Connes \cite[Chapter IV]{Connes:book} (see also \cite[Introduction, Section 4]{Connes:book} for an overview), the role of the classical differential $\ud b$ of a function $b$ is taken by the commutator of $b$ (identified with the operator of pointwise multiplication $b:f\mapsto bf$) with a suitable fixed operator $T$:
\begin{equation*}
  \dbar b:=[T,b]:=Tb-bT:\quad f\mapsto T(bf)-bTf.
\end{equation*}
Regarding $\dbar b$ as an operator on a Hilbert space $\mathcal H$, it is of particular interest in this theory to understand when the singular values $a_n(\dbar b)$ exhibit asymptotic decay of the form $n^{-\frac{1}{p}}$, as this allows one to compute the Dixmier traces of $\abs{\dbar b}^p$, which are interpreted as quantised analogues of the integral of the classical infinitesimal $\abs{\!\ud b}^p$. For a bounded linear operator $R$ on a Hilbert space $\mathcal H$, we recall that its $n$th singular value (or approximation number, hence the notation) is defined by
\begin{equation}\label{eq:an(R)}
  a_n(R):=a_n(R;\mathcal H)
  :=\inf\{\Norm{R-F}{\bddlin(\mathcal H)}:\operatorname{rank}F<n\},
\end{equation}
where $\Norm{\ }{\bddlin(\mathcal H)}$ denotes the operator norm on $\mathcal H$, and
 $\operatorname{rank}F<n$ means that $F$ is a linear combination of less than $n$ rank-one operators
 \begin{equation*}
  e\otimes h:\quad f\mapsto e\pair{h}{f},
\end{equation*}
where $e,h\in\mathcal H$.
 
Turning to concrete models with a specific choice of the operator $T$ on a Hilbert space of square-integrable functions $\mathcal H=L^2(\mu)$ with respect to some measure $\mu$, the desired asymptotics take the form
\begin{equation} \label{eq:specAs}
  \lim_{n\to\infty } a_n([T,b];L^2(\mu))\cdot n^{\frac1d} 
  \sim \Norm{b}{\dot W^{1,d}(\nu)}
  :=\Norm{\nabla b}{L^p(\nu)},
\end{equation}
where $d$ is a relevant critical index, and the second measure $\nu$ on the right is often but not always the same as the measure $\mu$ on the left.
So far, a behaviour like \eqref{eq:specAs} has been confirmed in a handful of situations, where $T$ is either the sign of a Dirac operator on $\R^n$ \cite{FSZ:23}, a classical Riesz transform $R_j=\partial_j(-\Delta)^{-\frac12}$ on $\R^n$ (a case essentially contained in \cite{FSZ:23}, as observed in \cite[Theorem C.1(iii)]{FLSZ}), the fractional Laplacian $(-\Delta)^{\frac s2}$ on $\R^n$ \cite{FSZ:24}, or a Bessel--Riesz transform $R_{\lambda,j}=\partial_j(-\Delta_\lambda)^{-\frac12}$ related to the Bessel Laplacian $\Delta_\lambda:=\Delta+2\lambda x_{n+1}^{-1}\partial_{n+1}$ on $(\R^{n+1}_+,x_{n+1}^{2\lambda}\ud x)$ \cite{FLSZ}. A related trace formula has also been obtained for Riesz transforms on Heisenberg groups in \cite{FLMSZ}. Although outside the scope of our present investigation, we mention that analogous results are also available on the quantum torus $\T_\theta^d$ \cite{MSX:tori,SXZ} and the quantum Euclidean space $\R_\theta^d$ \cite{MSX:Euc}; in these cases, $\dot W^{1,d}(\nu)$ in \eqref{eq:specAs} should also be interpreted as a suitable quantum Sobolev space.

In the meantime, the uniform, non-asymptotic version of \eqref{eq:specAs}, namely
\begin{equation}\label{eq:weakSd}
  \Norm{[T,b]}{S^{d,\infty}(L^2(\mu))}
  \sim \Norm{b}{\dot W^{1,d}(\nu)},\qquad
  \Norm{R}{S^{p,\infty}(\mathcal H)}
  :=\sup_{n\geq 1} n^{\frac1p} a_n(R;\mathcal H),
\end{equation}
as well as related estimates for other Schatten--Lorentz norms with $p,q\in(0,\infty)$,
\begin{equation}\label{eq:Spq-def}
  \Norm{R}{S^{p,q}(\mathcal H)}:=
  \Norm{\{a_n(R;\mathcal H)\}_{n=1}^\infty}{\ell^{p,q}}
  :=\Big(\sum_{n=1}^\infty n^{\frac qp-1}a_n(R;\mathcal H)^q\Big)^{\frac 1q},
\end{equation}
have been proved in a much broader generality beyond the prominent special cases listed above. After being established in several concrete settings \cite{CST,FLMSZ,FLSZ,LXY,LMSZ} first, the uniform estimate \eqref{eq:weakSd} is now known to hold for any non-generate singular integral operators $T$ over metric measure spaces of controlled geometry \cite{Hyt:Osc,Hyt:Sp}; an abstract framework covering the previous concrete cases \cite{CST,FLMSZ,LXY,LMSZ} of \eqref{eq:weakSd} with equal measures $\mu=\nu$ was first obtained in \cite{Hyt:Sp}, and this was further extended in \cite{Hyt:Osc} to the case of two different $\mu\neq\nu$, incorporating also the results of \cite{FLSZ} on the Bessel setting into the general theory.

The benefits of such a general theory are evident. Instead of treating each situation of interest on a case-by-case basis, one can simply appeal to a general result, which also makes the underlying assumptions more transparent. However, until now, the asymptotic behaviour of the singular values was only available in the distinguished special cases listed above. Our present contribution improves the situation by providing the following general statement, where $\dot M^{1,d}(\nu)$ is the so-called Haj\l{}asz--Sobolev space from \cite{Hajlasz:96}, a suitable replacement in abstract spaces of the classical $\dot W^{1,d}$ from \eqref{eq:specAs}. We refer the reader to Section \ref{sec:prelim} for the full details of this and other assumptions in Theorem \ref{thm:main} below.

\begin{theorem}\label{thm:main}
Let $(X,\rho)$ be a metric space with two doubling measures $\mu,\nu$ that satisfy the $A_\infty$ relation. Suppose further that $(X,\rho,\nu)$ is Ahlfors $d$-regular for some $d\in(1,\infty)$ and supports the $q$-Poincar\'e inequality for some $q\in[1,d)$. Let $T$ be a strongly non-degenerate $\mu$-singular integral and $b\in L^1_{\loc}(\mu)$. Suppose that either
\begin{enumerate}[\rm(1)]
  \item\label{it:main-reg} $T\in\bddlin(L^2(\mu))$ is H\"older $\eta$-regular with $\eta>(1-\frac{d}{2})_+$, or
  \item\label{it:main-new} $\mu=\nu$ and $d>2$.
\end{enumerate}
Then $[T,b]\in S^{d,\infty}(L^2(\mu))$ if and only if $b\in\dot M^{1,d}(\nu)$, and in this case
\begin{subequations}\label{eq:main}
\begin{align}
  \Norm{b}{\dot M^{1,d}(\nu)}
  &\sim\Norm{[T,b]}{S^{d,\infty}(L^2(\mu))}:=\sup_{n\geq 1} a_n([T,b];L^2(\mu))\cdot n^{\frac1d}
  \label{eq:main-bound} \\
  &\sim\liminf_{n\to\infty } a_n([T,b];L^2(\mu))\cdot n^{\frac1d} 
  \sim\limsup_{n\to\infty } a_n([T,b];L^2(\mu))\cdot n^{\frac1d} .
  \label{eq:main-asymp}
\end{align}
\end{subequations}
\end{theorem}

Some comments are in order:
\begin{enumerate}[\rm(a)]
  \item It should be admitted right away that Theorem \ref{thm:main} does not strictly contain the earlier spectral asymptotic results cited above, where, in place of \eqref{eq:main-asymp}, the corresponding limit exists and the comparison ``$\sim$'' between the left and right-hand sides of \eqref{eq:main} becomes an identity ``$=$'' with an explicit constant. However, relaxing the existence of the limit to the comparable size of the upper and lower limits (which still says that the singular values behave essentially like $n^{-\frac1d}$ as $n\to\infty$), Theorem \ref{thm:main} applies to a far greater class of operators than any of the previous asymptotic results of this type. Aside from the abstract setting of Theorem \ref{thm:main}, the asymptotic relation \eqref{eq:main-asymp} seems to be new even in the concrete case of Riesz transforms on Heisenberg groups, and more generally for non-degenerate, highly regular singular integral operators on Carnot groups, for which the uniform estimate \eqref{eq:main-bound} has been obtained in \cite{FLMSZ} and \cite{LXY}, respectively.
  
  \item The $A_\infty$ relation of $\mu$ and $\nu$ is trivially satisfied for equal measures $\mu=\nu$, and this case already covers many important examples (like the Heisenberg and Carnot groups just mentioend), but the setting with two different measures further increases the scope of the result, notably, to include the Bessel setting \cite{FLSZ}, as observed in \cite{Hyt:Osc}. For the Bessel--Riesz transforms, a sharper version of \eqref{eq:main-asymp} with exact limit is contained in \cite{FLSZ}, but Theorem \ref{thm:main} also applies to any other non-degenerate, H\"older $\eta$-regular singular integrals in the Bessel setting.

  \item The uniform bound \eqref{eq:main-bound} is simply a restatement of \cite[Theorem 1.4]{Hyt:Osc} in case \eqref{it:main-reg}, but appears to be new in case \eqref{it:main-new}. Note that this latter case does {\em not} assume any regularity of the kernel, which seems to be new for the weak-type Schatten bounds even in $\R^d$.
  
  The proof of this case is perhaps even more interesting than the result. All previous proofs of ``$\gtrsim$'' in \eqref{eq:main-bound} depend on somewhat delicate decompositions of the operator $T$: either the {\em nearly weakly orthogonal} (NWO) expansions of Rochberg--Semmes \cite{RS:NWO} (first obtained by local Fourier series expansions on $\R^d$ in \cite{RS:NWO} and later with the help of so-called Alpert bases on Carnot groups in \cite{LXY}), or the {\em dyadic representation theorem} from \cite{Hyt:A2} (first applied to Schatten bounds on $\R^d$ in \cite{WZ:median} and then on doubling metric measure spaces in \cite{Hyt:Sp}).
  
  For related strong-type Schatten $S^p$ estimates with $p>2$, there is a much simpler approach due to Janson--Wolff \cite{JW:82}, originally in $\R^d$, and extended to other settings e.g.\ in \cite[Sec.~4.2]{FLL:23} and \cite[Sec.~2.2]{LXY}. However, the fact that a slight elaboration of their approach can also yield $S^{d,\infty}$ bounds at the critical index, under the assumptions of case \eqref{it:main-new}, seems to have been overlooked before. Note that this case covers, in particular, all Heisenberg groups and more generally all nontrivial Carnot groups (i.e., those that do not reduce to Euclidean spaces), as these are Ahlfors $d$-regular with $d\geq 4$ (cf.~\cite[p.~77]{LXY}). Thus, we obtain a third, different proof of the upper bound in \cite[Theorem 1.2]{LXY}, where the original one in \cite{LXY} was based on Alpert bases, the second one in \cite{Hyt:Sp} on dyadic representation, and the current one on the said elaboration of the Janson--Wolff method.
\end{enumerate}

We record some corollaries of Theorem \ref{thm:main}:

\begin{corollary}\label{cor:new-cutoff}
Let the assumptions of Theorem \ref{thm:main} be satisfied. If 
\begin{equation}\label{eq:lim-cutoff}
  \lim_{n\to\infty} n^{\frac1d}a_n([T,b])=0,
\end{equation}
or more generally if
\begin{equation}\label{eq:new-cutoff}
  \limsup_{n\to\infty}n^{\frac1d}a_n([T,b])<\infty,\qquad
  \liminf_{n\to\infty} n^{\frac1d}a_n([T,b])=0,
\end{equation}
then $b$ is equal to constant almost everywhere.
\end{corollary}

Corollary \ref{cor:new-cutoff} sharpens the Janson--Wolff type cut-off phenomenon (first discovered in \cite{JW:82} for $X=\R^d$, and extended to the setting of Theorem \ref{thm:main} in \cite{Hyt:Osc,Hyt:Sp})
\begin{equation*}
   [b,T]\in S^p(L^2(\mu))\quad
   \Leftrightarrow\quad b=\text{constant a.e.},\qquad\text{for }p\in(0,d],
\end{equation*}
since borh \eqref{eq:lim-cutoff} and \eqref{eq:new-cutoff} are weaker than
$\sum_{n=1}^\infty a_n([T,b])^d<\infty.$ In the special cases, where $T$ is a Dirac operator on $\R^d$ or one of the Bessel--Riesz transforms, similar corollaries are recorded in \cite[Corollary 1.5]{FLSZ} and \cite[Corollary 1.3]{FSZ:23}; for such conclusions, the possible lack of existence of the actual limit in \eqref{eq:main-asymp} is not an issue. 

Another corollary of Theorem \ref{thm:main} is the following. The terminology related to traces is explained in Section \ref{sec:trace}.

\begin{corollary}\label{cor:trace}
Let the assumptions of Theorem \ref{thm:main} be satisfied.
If $$[T,b]\in S^{d,\infty}(L^2(\mu))$$ and $\tau$ is a positive normalised trace on $S^{1,\infty}(L^2(\mu))$, then
\begin{equation*}
  \tau(\abs{[T,b]}^d)\sim\Norm{b}{\dot M^{1,d}(\nu)}^d.
\end{equation*}
\end{corollary}

When specialised to the concrete cases treated before, Corollary \ref{cor:trace} is in the same spirit but not strictly comparable to \cite[Theorem 1.2]{FLMSZ}, \cite[Corollary 1.4]{FLSZ}, or \cite[Theorem 2]{LMSZ}, which apply to every {\em continuous} normalised trace on $S^{1,\infty}(L^2(\mu))$. Thus, the continuity of the trace assumed in these results is replaced by the positivity assumption in Corollary \ref{cor:trace}. We note that this assumption covers, in particular, all Dixmier traces \cite[p.~xviii]{LSZ:book}. As in Theorem \ref{thm:main}, the main novelty of Corollary \ref{cor:trace} is dealing with a much larger class of operators $T$ than before.

\subsection{Key ingredient of the proof: asymptotic interpolation}

Let us briefly discuss the proof of Theorem \ref{thm:main}. Given that \eqref{eq:main-bound} is already known (in case \eqref{it:main-reg} of Theorem \ref{thm:main}) from \cite{Hyt:Osc,Hyt:Sp}, and we have the obvious inequalities $\liminf\leq\limsup\leq\sup$, what remains to prove, in order to obtain \eqref{eq:main}, is the bound
\begin{equation}\label{eq:intro-to-prove}
   \Norm{b}{\dot M^{1,d}(\nu)}\lesssim\liminf_{n\to\infty} a_n([T,b];L^2(\mu))\cdot n^{\frac1d}.
\end{equation}
This can be thought of as a quantised version of the following function space norm inequality due to \cite{CST,Frank} for $\ud\nu=\ud x$ on $\R^d$ and \cite{HK:W1p} in general:
\begin{equation}\label{eq:intro-Frank}
   \Norm{b}{\dot M^{1,d}(\nu)}\lesssim
   \liminf_{\lambda\to 0}\lambda\cdot\nu_d(\{(x,t)\in X\times(0,\infty):
     \operatorname{osc}_\nu(b;B(x,t))>\lambda\})^{\frac1d},
\end{equation}
where $\ud\nu_d(x,t):=\ud\nu(x)t^{-d-1}\ud t$ and
\begin{equation*}
  \operatorname{osc}_\nu(b;B):=\inf_{c\in\C}\fint_B\abs{b-c}\ud\nu
\end{equation*}
is the usual mean oscillation over a ball $B$. Thus, a potential strategy to \eqref{eq:intro-to-prove} is to bound the right-hand side of \eqref{eq:intro-Frank} in terms of the right-hand side of \eqref{eq:intro-to-prove}. It turns out that this is essentially possible, but with an interesting twist.

Under the usual assumptions of the Marcinkiewiz interpolation theorem concerning a sublinear operator $S$ (see Theorem \ref{thm:liminf} for details), we will obtain new kinds of asymptotic conclusions of the form
\begin{equation}\label{eq:intro-liminf}
  \liminf_{\lambda\to 0}\lambda\cdot\nu(\{\abs{Sf}>\lambda\})^{\frac1p}
  \lesssim\Big(\liminf_{\lambda\to 0}\lambda\cdot\mu(\{\abs{f}>\lambda\})^{\frac1p}\Big)^\theta
    \Norm{f}{L^{p,\infty}(\mu)}^{1-\theta}
\end{equation}
for some $\theta\in(0,1)$ depending on $p$ and the other exponents in the interpolation assumptions, and the right-hand side of \eqref{eq:intro-liminf} may also be replaced by its noncommutative analogue involving $S^{p,\infty}$ in place of $L^{p,\infty}$ (Corollary \ref{cor:liminf}). We show in Example \ref{ex:Hardy} that this is the best kind of conclusion that one can hope for in general; in particular, a pure $\liminf$ estimate with $\theta=1$ in \eqref{eq:intro-liminf} is false in general, even for very natural operators $S$.

But \eqref{eq:intro-liminf} turns out to be enough for our application.
Namely, it follows that
\begin{subequations}\label{eq:intro-bootstrap}
\begin{align}
      \Norm{b}{\dot M^{1,d}(\nu)}
      &\overset{\text{\eqref{eq:intro-Frank}}}\lesssim
   \liminf_{\lambda\to 0}\lambda\cdot
   \nu_d(\{(x,t): \operatorname{osc}_\nu(b;B(x,t))>\lambda\})^{\frac1d}
   \label{eq:M<Osc}
    \\
   &\overset{\text{\eqref{eq:intro-liminf}}}{\lesssim}\Big(\liminf_{n\to\infty} a_n([T,b];L^2(\mu))\cdot n^{\frac1d}\Big)^\theta
   \Norm{[T,b]}{S^{d,\infty}(L^2(\mu))}^{1-\theta}
   \label{eq:Osc<[T,b]}
    \\
   &\overset{\text{\eqref{eq:main-bound}}}{\lesssim}\Big(\liminf_{n\to\infty} a_n([T,b];L^2(\mu))\cdot n^{\frac1d}\Big)^\theta   \Norm{b}{\dot M^{1,d}(\nu)}^{1-\theta},
   \label{eq:[T,b]<M}
\end{align}
\end{subequations}
and we obtain the desired \eqref{eq:intro-to-prove} by rearranging. In this sketch, the application of ``\eqref{eq:intro-liminf}'' in the middle step is abbreviation for ``a non-commutative extension of \eqref{eq:intro-liminf}, applied to a suitable operator $S$ that is implicit in the equivalence \eqref{eq:main-bound} in \cite{Hyt:Osc,Hyt:Sp} and worked out explicitly in Section \ref{sec:implicit}.''

Thus, the key new tool to reach our goal is the asymptotic interpolation estimate \eqref{eq:intro-liminf}, which allows us to pass not only global size but also asymptotic information of a function through a sublinear operator under the usual assumptions of Marcinkiewicz interpolation. This seems particularly curious when ``asymptotic information'' is understood in terms of $\liminf$, as above, which seems conceptually far from norm-like quantities that are typically controlled by various interpolation theorems: for example, two nonnegative sequences with $\liminf a_n=0=\liminf b_n$ can very easily have $\liminf(a_n+b_n)>0$. A $\limsup$ variant of \eqref{eq:intro-liminf} (also contained in Theorem \ref{thm:liminf}) is somewhat closer in spirit to classical interpolation theorems, and slightly easier to prove, but it would only give a weaker form of Theorem \ref{thm:main}.

Besides its key role behind Theorem \ref{thm:main}, we believe that the asymptotic interpolation bound \eqref{eq:intro-liminf} (and its variants in Theorem \ref{thm:liminf}, including ones with $\lambda\to\infty$) will be useful in other applications, especially given the increasing recent role of such asymptotic weak-type quantities in the ``surprising formulas for Sobolev norms'' after Brezis, Van Schaftingen, and Yung \cite{BVY}; see \cite{BSVY,DM:22,ZLYYZ} for a sample of variants and extensions. With estimates like \eqref{eq:intro-liminf}, these asymptotic quantities can also be passed through operators in a controlled way, and \eqref{eq:intro-bootstrap} serves as a model on how to eliminate the parameter $\theta$ from such estimates under appropriate additional information. In essence, \eqref{eq:intro-bootstrap} provides a template for proving results of the following type (under relevant technical assumptions that need to be supplied):
\begin{quote}
  ``If the $L^{p,\infty}$ norms of two objects are equivalent, and one of them is also equivalent to the corresponding asymptotic quantity, then the other one satisfies a similar asymptotic equivalence as well.''
\end{quote}

The rest of the paper is organised as follows. In Section \ref{sec:prelim}, we collect the definitions of the various notions that were already used in this Introduction, and we also detail the proofs of Corollaries \ref{cor:new-cutoff} and \ref{cor:trace}, assuming Theorem \ref{thm:main}.  In Section \ref{sec:interpol}, we prove the new asymptotic interpolation theorem, the key tool behind Theorem \ref{thm:main}. In Section \ref{sec:implicit}, we identify the relevant sublinear operators (implicit in some previous related works) to which we apply the asymptotic interpolation theorem to derive a preliminary version of \eqref{eq:Osc<[T,b]} with a suitable discrete weak-type oscillation norm in place of that in \eqref{eq:M<Osc}. The relation of this discrete oscillation norm with the one in \eqref{eq:M<Osc} is then sorted out in Section \ref{sec:Osc}, where we complete the proof of case \eqref{it:main-reg} of Theorem \ref{thm:main}. The modifications for case \eqref{it:main-new} are handled in Section \ref{sec:upper}, where we present the extension of the Janson--Wolff method to the weak-type commutator bounds at the critical index.

\section{Definitions and preliminaries}\label{sec:prelim}

In this section, we recall the relevant definitions that we already used in Theorem \ref{thm:main}. We also prove Corollaries \ref{cor:new-cutoff} and \ref{cor:trace}, assuming Theorem \ref{thm:main}.

\subsection{Metric measure spaces}

A measure $\mu$ on a metric measure space $(X,\rho)$ is called doubling, is all metric balls $B(x,r)$ have finite positive measure and
\begin{equation*}
  \mu(B(x,2r))\lesssim \mu(B(x,r))
\end{equation*}
Two such measures $\mu,\nu$ are said to satisfy the $A_\infty$-relation, if there are constants $\eps,\delta>0$ such that for all metric balls $B$ and their measurable subsets $E\subset B$, we have
\begin{equation}\label{eq:Ainfty}
  \mu(E)\leq\eps\mu(B)\quad\Rightarrow\quad\nu(E)\leq(1-\delta)\nu(B).
\end{equation}
It is easy to check that this relation is symmetric with respect to $\mu$ and $\nu$. When two measures satisfy this relation, they are mutually absolutely continuous, and hence we can speak of ``almost everywhere'' unambiguously, even in the presence of these two measures.

A space $(X,\rho,\nu)$ is called Ahlfors $d$-regular if $\nu(B(x,r))\sim r^d$ for all metric balls $B(x,r)$.
If a metric space $(X,\rho)$ supports such a measure $\nu$ (as it is assumed in Theorem \ref{thm:main}), then it has the following geometric property. Let $(x_i)_{i=1}^N$ be $r$-separated points in a ball $B(x,R)$, where $r\leq R$. Then the number of such points satisfies $N\lesssim(R/r)^d$: the balls $B(x_i,\frac12 r)\subseteq B(x,R+\frac12 r)$ are disjoint, so that
\begin{equation*}
  R^d\sim(R+\tfrac12 r)^d\sim \nu(B(x,R+\tfrac12))\geq\sum_{i=1}^N\nu(B(x_i,\tfrac12r))
  \sim N(\tfrac12r)^d\sim N r^d.
\end{equation*}

A metric measure space $(X,\rho,\nu)$ is said to satisfy the $q$-Poincar\'e inequality if there is a constant $\lambda\in[1,\infty)$ such that all Lipschitz functions $f$ satisfy
\begin{equation*}
  \inf_{c\in\C}\fint_{B(x,r)}\abs{f-c}\ud\nu
  \lesssim r\Big(\fint_{B(x,\lambda r)}(\operatorname{lip}f)^q\ud\nu\Big)^{\frac1q},
\end{equation*}
where
\begin{equation*}
  \operatorname{lip}f(y):=\liminf_{t\to 0}\sup_{z\in B(y,t)}\frac{\abs{f(y)-f(z)}}{t}.
\end{equation*}

A function $h:X\to[0,\infty]$ is called a Haj\l{}asz upper gradient of $f$ if
\begin{equation*}
  \abs{f(x)-f(y)}\leq\rho(x,y)(h(x)+h(y))
\end{equation*}
for almost all $x,y\in X$. The homogeneous Haj\l{}asz--Sobolev norm of $f$ is defined by
\begin{equation*}
  \Norm{f}{\dot M^{1,p}(\mu)}
  :=\inf\Big\{\Norm{h}{L^p(\mu)}: h\text{ is a Haj\l{}asz upper gradient of }f\Big\}.
\end{equation*}

\begin{lemma}\label{lem:f=0}
If $\Norm{f}{\dot M^{1,p}(\mu)}=0$, then $f=0$ almost everywhere.
\end{lemma}

\begin{proof}
By definition, we can find Haj\l{}asz upper gradients $h_n$ of $f$ with $\Norm{h_n}{L^p(\mu)}\to 0$. Passing to a subsequence if necessary, we can assume that $h_n\to 0$ pointwise almost everywhere. Hence
\begin{equation*}
  \abs{f(x)-f(y)}\leq\rho(x,y)(h_n(x)+h_n(y))\to 0,
\end{equation*}
and thus $f(x)=f(y)$ almost everywhere.
\end{proof}

\begin{proof}[Proof of Corollary \ref{cor:new-cutoff}]
It is evident that \eqref{eq:lim-cutoff} implies \eqref{eq:new-cutoff}, so it suffices to work under the assumption \eqref{eq:new-cutoff}. Let us denote $R:=[T,b]$ and $\mathcal H:=L^2(\mu)$. By \eqref{eq:new-cutoff}, if follows in particular that $a_{n_0}(R)<\infty$ for some $n_0$. Thus, by definition \eqref{eq:an(R)}, there is $F\in\bddlin(\mathcal H)$ with $\operatorname{rank}F<n_0$ such that $\Norm{R-F}{\bddlin(\mathcal H)}<a_{n_0}(R)+1$, and hence $\Norm{R}{\bddlin(\mathcal H)}\leq\Norm{F}{\bddlin(\mathcal F)}+a_{n_0}(R)+1<\infty$. Now, directly from \eqref{eq:new-cutoff}, we know that $\sup_{n>N}a_n(R)\cdot n^{\frac1d}<\infty$ for some $N\in\N$. On the other other hand, by the previous observations, we also have
\begin{equation*}
  \sup_{1\leq n\leq N}a_n(R)\cdot n^{\frac1d}
  \leq \Norm{R}{\bddlin(\mathcal H)}N^{\frac1d}<\infty.
\end{equation*}
Hence $\sup_{n\geq 1}a_n(R)\cdot n^{\frac1d}<\infty$, and thus $R=[b,T]\in S^{d,\infty}(L^2(\mu))$.

Under this condition (and the assumptions of Theorem \ref{thm:main}), we have \eqref{eq:main}, and hence in particular
\begin{equation*}
  \Norm{b}{\dot M^{1,d}(\nu)}
  \sim \liminf_{n\to\infty } n^{\frac1d} a_n([T,b])=0.
\end{equation*}
The vanishing of $\Norm{b}{\dot M^{1,d}(\nu)}$ implies that $b$ is equal to constant almost everywhere by Lemma \ref{lem:f=0}, and this completes the proof of Corollary \ref{cor:new-cutoff}
\end{proof}

\subsection{Kernels}

A function $K$ defined on $\{(x,y)\in X\times X:x\neq y\}$
is said to be a {\em $\mu$-singular kernel} if
\begin{equation*}
  \abs{K(x,y)}\lesssim\frac{1}{V(x,y)},\qquad V(x,y):=\mu(B(x,\rho(x,y))).
\end{equation*}
A $\mu$-singular kernel is called
\begin{enumerate}[\rm(1)]
  \item {\em $\omega$-regular}, for a nonnegative nondecreasing function $\omega$, if
\begin{equation*}
  \abs{K(x,y)-K(x',y)}
  +\abs{K(y,x)-K(y,x')}
  \lesssim\frac{1}{V(x,y)}\omega\Big(\frac{\rho(x,x')}{\rho(x,y)}\Big)
\end{equation*}
for all $x,x',y\in X$ such that $\rho(x,x')\ll\rho(x,y)$;
 \item {\em H\"older $\eta$-regular}, if it is $\omega$-regular with $\omega(t)=t^\eta$;
  \item {\em non-degenerate}, if for every $x\in X$ and $r>0$, there exists $y\in X$ at distance $\rho(x,y)\sim r$ such that
\begin{equation*}
  \abs{K(x,y)}+\abs{K(y,x)}\gtrsim\frac{1}{V(x,y)};
\end{equation*}
  \item {\em strongly non-degenerate}, if there is a constant $\alpha>0$ such that for every $x_0\in X$ and $r>0$, there exists $y_0\in X$ at distance $\rho(x_0,y_0)\sim r$ and $v\in\C$ with $\abs{v}=1$ such that
\begin{equation*}
  \abs{K(x,y)}\gtrsim\frac{1}{V(x,y)},\qquad
  \arg(\bar v K(x,y))\leq\frac{\pi}{9}
\end{equation*}
either for all $(x,y)\in B(x_0,\alpha r)\times B(y_0,\alpha  r)$ or for all transposed pairs $(y,x)\in B(x_0,\alpha r)\times B(y_0,\alpha r)$.
\end{enumerate}

\begin{lemma}\label{lem:reg=strong}
Let $K$ be an $\omega$-regular $\mu$-singular kernel, where $\lim_{t\to 0}\omega(t)=0$. Then $K$ is non-degenerate if and only if it is strongly non-degenerate.
\end{lemma}

\begin{proof}
For a variant of the same definitions involving fractional (instead of singular) kernels, this is \cite[Lemma 10.2]{HW:frac}. The proof extends to singular kernels without changes; in fact, it formally a special case taking $\eps=0$ in the notation of \cite{HW:frac}, although only $\eps>0$ is explicitly considered there.
\end{proof}

Some results can proved for strongly non-degenerate kernels without assuming any regularity. Lemma \ref{lem:reg=strong} shows that they are not less general than corresponding results for regular non-degenerate kernels.

\subsection{Traces}\label{sec:trace}

We recall the some notions and results related to traces. These are only needed for Corollary \ref{cor:trace} and the related discussion.

Let $\mathcal J\subseteq\bddlin(\mathcal H)$ be a two-sided ideal (i.e., if $A\in\mathcal J$ and $S,T\in\bddlin(\mathcal H)$, then $SAT\in\mathcal J$).
Following \cite[Definition 1.2.10]{LSZ:book}, a {\em trace} on $\mathcal J$ is a linear functional $\tau:\mathcal J\to\C$ such that $\tau(UAU^*)=\tau(A)$ for all $A\in \mathcal J$ and all unitary operators $U\in\bddlin(\mathcal H)$. A trace $\tau$ is called
\begin{enumerate}[\rm(1)]
  \item {\em positive} if $\tau(A)\geq 0$ whenever $A\geq 0$ in the sense of positive operators;
  \item {\em singular} if $\tau(A)=0$ whenever $A$ is a finite-rank operator \cite[Definition 1.3.2]{LSZ:book}.
\end{enumerate}

Following \cite[page 186]{LSZ:book}, a trace $\tau$ on $\mathcal J=S^{1,\infty}(\mathcal H)$ is called {\em normalised}, if 
\begin{equation}\label{eq:normalised}
  \tau\Big(\sum_{n=1}^\infty \frac{1}{n}e_n\otimes e_n\Big)=1,
\end{equation}
for some (and thus, by unitary equivalence, every) orthonormal basis $(e_n)_{n=1}^\infty$.

\begin{proposition}[\cite{LSZ:book}, Corollary 3.5.4]\label{prop:singular}
Every trace on $S^{1,\infty}(\mathcal H)$ is singular.
\end{proposition}

We can now give:

\begin{proof}[Proof of Corollary \ref{cor:trace}]
Since $[T,b]\in S^{d,\infty}(L^2(\mu))$ by assumption, we have
\begin{equation*}
  \abs{[T,b]}^d
  =\sum_{n=1}^\infty a_n([T,b])^d e_n\otimes e_n
\end{equation*}
for some orthonormal basis $(e_n)_{n=1}^\infty$. By \eqref{eq:main}, it follows that
\begin{equation*}
     0\leq a_n([T,b])\cdot n^{\frac1d}
     \begin{cases} \leq c_1\Norm{b}{\dot M^{1,d}(\nu)}, & \text{for all }n\geq 1, \\
     \geq c_0\Norm{b}{\dot M^{1,d}(\nu)}, &\text{for all }n>N,\end{cases}
\end{equation*}
for some constants $c_0,c_1$ and $N\in\N$. Hence, in the sense of positive operators, we have
\begin{equation*}
   c_0^d\Norm{b}{\dot M^{1,d}(\nu)}^d\sum_{n>N}\frac 1n e_n\otimes e_n
   \leq \abs{[T,b]}^d
   \leq c_1^d\Norm{b}{\dot M^{1,d}(\nu)}^d\sum_{n=1}^\infty \frac 1n e_n\otimes e_n.
\end{equation*}
Applying the trace $\tau$, it follows from the assumed positivity that
\begin{equation*}
   c_0^d\Norm{b}{\dot M^{1,d}(\nu)}^d\tau\Big(\sum_{n>N}\frac 1n e_n\otimes e_n\Big)
   \leq \tau(\abs{[T,b]}^d)
   \leq c_1^d\Norm{b}{\dot M^{1,d}(\nu)}^d
   \tau\Big(\sum_{n=1}^\infty \frac 1n e_n\otimes e_n\Big).
\end{equation*}
By the definition of a normalised trace, \eqref{eq:normalised}, the right-most factor above is equal to~$1$. Combining this with the fact that every trace on $S^{1,\infty}(L^2(\mu))$ is singular (Proposition \ref{prop:singular}), it also follows that
\begin{equation*}
  \tau\Big(\sum_{n>N} \frac1n e_n\otimes e_n\Big)
  =\tau\Big(\sum_{n=1}^\infty \frac 1n e_n\otimes e_n\Big)
  -\tau\Big(\sum_{n=1}^N \frac1n e_n\otimes e_n\Big)
  =1-0=1.
\end{equation*}
Substituting back, we obtain
\begin{equation*}
  c_0^d\Norm{b}{\dot M^{1,d}(\nu)}^d
   \leq \tau(\abs{[T,b]}^d)
   \leq c_1^d\Norm{b}{\dot M^{1,d}(\nu)}^d,
\end{equation*}
which is precisely the claim of Corollary \ref{cor:trace}.
\end{proof}

\section{Asymptotic interpolation}\label{sec:interpol}

The goal of this section is to prove the following variant of the Marcinkiewicz interpolation theorem involving asymptotic quantities (lower and upper limits) in place of the usual supremums defining weak-type norms. This is the key tool in our proof of Theorem \ref{thm:main}, but it may also have independent interest as a new kind of interpolation theorem. This section is self-contained and can also be read independently of the rest of the paper. After the proof of Theorem \ref{thm:liminf}, we demonstrate the optimality of its conclusions in Example \ref{ex:Hardy}. At the end of the section, we provide a noncommutative version of Theorem \ref{thm:liminf} in Corollary \ref{cor:liminf}; this is also needed for the proof of Theorem \ref{thm:main}.

\begin{theorem}\label{thm:liminf}
Let $0<q<p<r\leq\infty$ and 
\begin{equation*}
  T:L^{q}(\mu)+L^{r}(\mu)\to L^{q,\infty}(\nu)+L^{r,\infty}(\nu)
\end{equation*}
be a quasi-sublinear operator, i.e., for some fixed $K\in[1,\infty)$,
\begin{equation}\label{eq:quasi-sub}
  \abs{T(f+g)}\leq K(\abs{Tf}+\abs{Tg}),
\end{equation}
such that $T$ has bounded restrictions
\begin{equation*}
  T:L^{q}(\mu)\to L^{q,\infty}(\nu),\quad
  T:L^{r}(\mu)\to L^{r,\infty}(\nu),
\end{equation*}
where $L^{\infty,\infty}:=L^\infty$.
If $f\in L^{p,\infty}(\mu)$ and $\omega\in\{0,\infty\}$, then
\begin{subequations}\label{eq:liminfsup}
\begin{align}
  \liminf_{\lambda\to\omega}
    \lambda\cdot\nu(\abs{Tf}>\lambda)^{\frac1p}
    &\lesssim \Big(\liminf_{\lambda\to\omega}
    \lambda\cdot\mu(\abs{f}>\lambda)^{\frac1p}\Big)^\theta \Norm{f}{L^{p,\infty}(\mu)}^{1-\theta},
    \label{eq:liminf} \\
    \limsup_{\lambda\to\omega}
    \lambda\cdot\nu(\abs{Tf}>\lambda)^{\frac1p}
    &\lesssim \Big(\limsup_{\lambda\to\omega}
    \lambda\cdot\mu(\abs{f}>\lambda)^{\frac1p}\Big)^{\Theta(\omega)} 
    \Norm{f}{L^{p,\infty}(\mu)}^{1-\Theta(\omega)},
    \label{eq:limsup}
\end{align}
\end{subequations}
where
\begin{subequations}\label{eq:thetas}
\begin{align}
  \theta &=\frac{(p-q)(r-p)}{p(r-q)}\qquad
  \Big(:=\frac{p-q}{p}\quad\text{if}\quad r=\infty\Big)
  \label{eq:theta} \\
  \Theta(\omega)
  &=\begin{cases}\displaystyle 
  \frac{(p-q)r}{p(r-q)}\quad \Big(:=\frac{p-q}{p}\quad\text{if}\quad r=\infty\Big), &
  \text{if}\quad  \omega=0, \\ \mbox{} \\
  \displaystyle\frac{r-p}{r-q}\qquad
  \Big(:=1\quad\text{if}\quad r=\infty\Big), & \text{if}\quad \omega=\infty.\end{cases}
  \label{eq:Theta}
\end{align}
\end{subequations}
If $T$ is a linear operator, then one can always take $\Theta(\omega)=1$ in \eqref{eq:limsup}.
\end{theorem}

\begin{remark}
Under the same assumptions, the classical Marcinkiewicz interpolation theorem (see \cite[Theorem 2.2.3]{HNVW1}) guarantees (among other things) that
\begin{equation*}
  \Norm{Tf}{L^{p,\infty}(\nu)}\lesssim\Norm{f}{L^{p,\infty}(\mu)}
  :=\sup_{\lambda>0}\lambda\cdot\mu(\abs{f}>\lambda)^{\frac1p}.
\end{equation*}
The point of \eqref{eq:liminfsup} is that, if we are willing to estimate a smaller asymptotic quantity (liminf or limsup instead of sup) on the left, then we can also replace the upper bound on the right by a related smaller quantity. However, this replacement is not linear: in general, an asymptotic quantity on the left cannot be controlled linearly by the similar asymptotic quantity on the right. While such a simpler bound is achieved for linear operators in the limsup case, we will show in Example \ref{ex:Hardy} that the liminf estimate, even for linear operators, cannot be better than that established in Theorem \ref{thm:liminf} in general. As we will see in the proof of Theorem \ref{thm:liminf}, the simpler form of the limsup estimate for linear operators can be attributed to the fact that $\limsup_{\lambda\to\omega}\lambda\cdot\mu(\abs{f}>\lambda)$ actually coincides with the quasi-norm of the quotient space of $L^{p,\infty}(\mu)$ with a certain subspace $L^{p,\infty}_\omega(\mu)$. Thus, \eqref{eq:limsup} is still quite close in the spirit with typical interpolation theorems in the sense of providing a quasi-norm estimate between some intermediate spaces. On the other hand, there seems to be no way of interpreting liminf in a similar way: since two nonnegative sequences with $\liminf a_n=0=\liminf b_n$ can easily have $\liminf (a_n+b_n)>0$, it follows that liminf is very far from what we expect from a (semi/quasi-)norm. This makes \eqref{eq:liminf} a more exotic addition to the old and established family of interpolation theorems, but it is precisely this exotic version that we need for our main application.

Theorem \ref{thm:liminf} shows that the limsup estimate \eqref{eq:limsup} holds with $\Theta(\omega)=1$ if {\em either} $T$ is linear {\em or} $r=\omega=\infty$. Example \ref{ex:Hardy} below shows that this conclusion cannot be extended to the liminf estimate \eqref{eq:liminf} even if {\em both} $T$ is linear {\em and} $r=\omega=\infty$.
Indeed, in the special case $q=1$ and $r=\infty$, which is perhaps the most frequently occurring situation for operators of harmonic analysis, Example \ref{ex:Hardy} shows that the exponent $\theta=1-\frac{1}{p}$ provided by Theorem \ref{thm:liminf} is the best possible for these parameters in general.

At the time of writing, we have not explored the possible sharpness of the exponents \eqref{eq:thetas} for other parameter values. For our main application, the precise exponent is actually irrelevant; any $\theta>0$ would suffice.
\end{remark}

\begin{proof}[Proof of Theorem \ref{thm:liminf}]
For $\lambda>0$, we denote
\begin{equation}\label{eq:f-lowhigh}
  f_\lambda:=f\cdot 1_{\{\abs{f}\leq\lambda\}},\quad f^\lambda:=f\cdot 1_{\{\abs{f}>\lambda\}}.
\end{equation}
For all $\alpha,\lambda>0$, we have $f=f^{\alpha\lambda}+f_{\alpha\lambda}$ and hence, by quasi-sublinearity \eqref{eq:quasi-sub},
\begin{equation*}
  \abs{Tf}\leq K(\abs{T(f^{\alpha\lambda})} + \abs{T(f_{\alpha\lambda})}).
\end{equation*}

Let first $r<\infty$. Then
\begin{equation}\label{eq:liminf-start}
\begin{split}
  \nu(\abs{Tf}>2K\lambda)
  &\leq \nu(\abs{T(f^{\alpha\lambda})}>\lambda) +\nu( \abs{T(f_{\alpha\lambda})}>\lambda ) \\
  &\lesssim\lambda^{-q}\Norm{f^{\alpha\lambda}}{L^{q}(\mu)}^q
    +\lambda^{-r}\Norm{f_{\alpha\lambda}}{L^{r}(\mu)}^r.
\end{split}
\end{equation}
Here
\begin{equation}\label{eq:liminf-low}
\begin{split}
  \Norm{f_{\alpha\lambda}}{L^r(\mu)}^r
  &=\int_0^\infty r t^{r-1}\mu(\abs{f}\cdot 1_{\{\abs{f}\leq\alpha\lambda\}}>t)\ud t \\
  &\leq\int_0^{\alpha\lambda} r t^{r-1}\mu(\abs{f}>t)\ud t \\
  &\leq\int_0^{\alpha\lambda} r t^{r-p-1} \ud t\times\sup_{s<\alpha\lambda}s^p\mu(\abs{f}>s) \\
  &\sim(\alpha\lambda)^{r-p} \sup_{s<\alpha\lambda}s^p\mu(\abs{f}>s).
\end{split}
\end{equation}
On the other hand, introducing another parameter $\beta>0$,
\begin{equation}\label{eq:liminf-high}
\begin{split}
  \Norm{f^{\alpha\lambda}}{L^q(\mu)}^q
  &=\int_0^\infty qt^{q-1}\mu(\abs{f}\cdot 1_{\{\abs{f}>\alpha\lambda\}}>\lambda)\ud t \\
  &\leq\int_0^{\beta\lambda} qr^{q-1}\mu(\abs{f}>\alpha\lambda)\ud t
     +\int_{\beta\lambda}^\infty qt^{q-1}\mu(\abs{f}>t)\ud t \\
   &\leq(\beta\lambda)^q\mu(\abs{f}>\alpha\lambda)+
   \int_{\beta\lambda}^\infty qt^{q-p-1}\ud t
   \times\sup_{s>\beta\lambda}s^p\mu(\abs{f}>s) \\
   &\lesssim (\beta\lambda)^q\mu(\abs{f}>\alpha\lambda)+
   (\beta\lambda)^{q-p}\sup_{s>\beta\lambda}s^p\mu(\abs{f}>s).
\end{split}
\end{equation}
Substituting \eqref{eq:liminf-low} and \eqref{eq:liminf-high} into \eqref{eq:liminf-start}, we obtain
\begin{equation}\label{eq:levelsets-sup}
\begin{split}
  (2K\lambda)^p &\nu(\abs{Tf}>2K\lambda) \\
  &\lesssim\lambda^{p-q}\Big((\beta\lambda)^q\mu(\abs{f}>\alpha\lambda)+
   (\beta\lambda)^{q-p}\sup_{s>\beta\lambda}s^p\mu(\abs{f}>s)\Big) \\
  &\qquad+\lambda^{p-r}(\alpha\lambda)^{r-p}\sup_{s<\alpha\lambda}s^p\mu(\abs{f}>s) \\
  &=\beta^q\alpha^{-p}(\alpha\lambda)^p\mu(\abs{f}>\alpha\lambda)
  +\beta^{q-p}\sup_{s>\beta\lambda}s^p\mu(\abs{f}>s)\ \\
  &\qquad+\alpha^{r-p}\sup_{s<\alpha\lambda}s^p\mu(\abs{f}>s).
\end{split}
\end{equation}
It is immediate from a change of variables that
\begin{equation}\label{eq:lim-term}
\begin{split}
  \liminf_{\lambda\to\omega}(\alpha\lambda)^p\mu(\abs{f}>\alpha\lambda)
   &=\liminf_{\lambda\to\omega}\lambda^p\mu(\abs{f}>\lambda)=:\underline{L}(\omega)^p, \\
   \limsup_{\lambda\to\omega}(\alpha\lambda)^p\mu(\abs{f}>\alpha\lambda)
   &=\limsup_{\lambda\to\omega}\lambda^p\mu(\abs{f}>\lambda)=:\overline{L}(\omega)^p.
\end{split}
\end{equation}
For the supremum terms in \eqref{eq:levelsets-sup}, we observe the uniform estimates and the limits
\begin{equation}\label{eq:sup-bounds}
\begin{split}
   \sup_{s<\alpha\lambda}s^p\mu(\abs{f}>s) &\leq N^p:=\Norm{f}{L^{p,\infty}(\mu)}^p,\quad
   \lim_{\lambda\to 0}\sup_{s<\alpha\lambda}s^p\mu(\abs{f}>s) =  \overline{L}(0)^p, \\
   \sup_{s>\beta\lambda}s^p\mu(\abs{f}>s) &\leq N^p,
   \phantom{:=\Norm{f}{L^{p,\infty}(\mu)}^p}
   \quad
   \lim_{\lambda\to\infty}\sup_{s>\beta\lambda}s^p\mu(\abs{f}>s) =  \overline{L}(\infty)^p.
\end{split}
\end{equation}

Taking $\liminf_{\lambda\to\omega}$ or $\limsup_{\lambda\to\omega}$ of both sides of \eqref{eq:levelsets-sup}, using the relevant versions of \eqref{eq:lim-term} and \eqref{eq:sup-bounds} available in each case, we obtain
\begin{equation} \label{eq:liminfsup-ab}
\begin{split}
  \liminf_{\lambda\to\omega}\lambda^p\nu(\abs{Tf}>\lambda)
  &\lesssim
  \beta^q\alpha^{-p}\underline{L}(\omega)^p+
   \beta^{q-p}N^p+\alpha^{r-p}N^p, \\
  \limsup_{\lambda\to\omega}\lambda^p\nu(\abs{Tf}>\lambda) 
  &\lesssim
  \beta^q\alpha^{-p}\overline{L}(\omega)^p+
  \begin{cases} \beta^{q-p}N^p+\alpha^{r-p}\overline{L}(\omega)^p, & \omega=0, \\
    \beta^{q-p}\overline{L}(\omega)^p+\alpha^{r-p}N^p, & \omega=\infty.
  \end{cases}
\end{split}
\end{equation}
In all cases of \eqref{eq:liminfsup-ab}, if $L\in\{\underline L(\omega),\overline L(\omega)\}$ satisfies $L=0$, then we can simply consider the limit $\alpha\to 0^+$ and $\beta\to\infty$ to conclude that the left-hand side of \eqref{eq:liminfsup-ab} vanishes, which is the desired estimate in this case.

So we proceed with the assumption that $0<L\leq N<\infty$. In each case, we then choose the parameters $\alpha$ and $\beta$ so that all three terms on the right of \eqref{eq:liminfsup-ab} are equal. For example, in the liminf case, we require that
\begin{equation*}
  \alpha^{-p}L^p=\beta^{-p}N^p,\qquad\beta^q L^p =\alpha^r N^p,
\end{equation*}
which has the solution $\alpha=(L/N)^{\frac{p-q}{r-q}}$ and $\beta=(N/L)^{\frac{r-p}{r-q}}$.
A substitution of these values into \eqref{eq:liminfsup-ab} gives the claimed bound \eqref{eq:liminf} with $\theta$ as in \eqref{eq:theta}. The proof of \eqref{eq:limsup} with $\Theta(\omega)$ as in \eqref{eq:Theta} is entirely similar.

This completes the proof of \eqref{eq:liminfsup} and \eqref{eq:thetas} for $r<\infty$.

\subsubsection*{Case $r=\infty$:}
If $r=\infty$, we estimate the contribution of the bounded part $f_{\alpha\lambda}$ as follows: Denoting by $C$ the norm of $T:L^\infty(\mu)\to L^\infty(\nu)$, it follows that
\begin{equation*}
  \Norm{T(f_{\alpha\lambda})}{L^\infty(\nu)}\leq C\Norm{f_{\alpha\lambda}}{L^\infty(\nu)}\leq C\alpha\lambda
  \leq\lambda
\end{equation*}
if we choose $\alpha=\max(C,1)^{-1}$. With this choice,
\begin{equation*}
  \nu(\abs{T(f_{\alpha\lambda})}>\lambda)=0.
\end{equation*}
Hence, in \eqref{eq:liminf-start}, the term involving $f_{\alpha\lambda}$ will vanish. We can then repeat the subsequent steps until \eqref{eq:liminfsup-ab} with the term coming from $f_{\alpha\lambda}$ removed. In place of \eqref{eq:liminfsup-ab}, we thus arrive at
\begin{equation} \label{eq:liminfsup-ab-infty}
\begin{split}
  \liminf_{\lambda\to\omega}\lambda^p\nu(\abs{Tf}>\lambda)
  &\lesssim
  \beta^q\underline{L}(\omega)^p+
   \beta^{q-p}N^p, \\
  \limsup_{\lambda\to\omega}\lambda^p\nu(\abs{Tf}>\lambda) 
  &\lesssim
  \beta^q\overline{L}(\omega)^p+
  \begin{cases} \beta^{q-p}N^p, & \omega=0, \\
    \beta^{q-p}\overline{L}(\omega)^p, & \omega=\infty;
  \end{cases}
\end{split}
\end{equation}
in addition to deleting the terms coming from $f_{\alpha\lambda}$, we also replaced $\alpha\sim 1$, since the parameter $\alpha$ was already fixed. In the case of $\limsup_{\lambda\to\omega}$, we note that the right-hand side of \eqref{eq:liminfsup-ab-infty} only has multiples of $\overline L(\omega)^p$, and we can also take $\beta=1$. In the remaining cases, the right-hand side of \eqref{eq:liminfsup-ab-infty} has the form $\beta^q L^p+\beta^{q-p}N^p$, where $L\in\{\underline L(\omega),\overline L(\omega)\}$. If $L=0$, we take $\beta\to\infty$, showing that the left-hand side of \eqref{eq:liminfsup-ab-infty} vanishes. If $0<L\leq N<\infty$, we make the two terms equal by choosing $\beta=N/L$. This immediately gives the claimed versions of \eqref{eq:liminfsup} and \eqref{eq:thetas} in this case, and completes the proof of Theorem \ref{thm:liminf} for sublinear operators.

\subsubsection*{Case of linear operators:}

Let us finally consider the case that $T$ is linear. For both $\sigma\in\{\mu,\nu\}$ and $\omega\in\{0,\infty\}$, we denote
\begin{equation}\label{eq:weakLp-omega}
  L^{p,\infty}_\omega(\sigma)
  :=\{f\in L^{p,\infty}(\sigma): \lim_{\lambda\to\omega}\lambda^p\sigma(\abs{f}>\lambda)=0\}.
\end{equation}
It is immediate that these are closed subspaces of $L^{p,\infty}(\sigma)$. By the part of the proof that we already completed, we know that $T:L^{p,\infty}_\omega(\mu)\to L^{p,\infty}_\omega(\nu)$. (Recall that this was a relatively immediate consequence of \eqref{eq:liminfsup-ab} or \eqref{eq:liminfsup-ab-infty}, and did not require the more delicate choice of the parameters $\alpha,\beta$.) This shows that the induced map between the quotient spaces,
\begin{equation*}
\begin{split}
  \tilde T:L^{p,\infty}(\mu)/L^{p,\infty}_\omega(\mu) &\to
  L^{p,\infty}(\nu)/L^{p,\infty}_\omega(\nu),\\
  [f]=f+L^{p,\infty}_\omega(\mu) &\mapsto
  [Tf]=Tf+L^{p,\infty}_\omega(\nu),
\end{split}
\end{equation*}
is well defined. It is standard that such induced maps satisfy $\Norm{\tilde T}{}\leq\Norm{T}{}$; this is well known for operators between Banach spaces, and the same proof can be repeated for quasi-Banach spaces (like $L^{p,\infty}(\mu)$) verbatim. Hence, it follows that
\begin{equation}\label{eq:quotient-bd}
  \Norm{[Tf]}{L^{p,\infty}(\nu)/L^{p,\infty}_\omega(\nu)}
  \lesssim\Norm{[f]}{L^{p,\infty}(\mu)/L^{p,\infty}_\omega(\mu)},
\end{equation}
where we have the usual quotient norms
\begin{equation}\label{eq:quotient-norm}
  \Norm{[f]}{L^{p,\infty}(\sigma)/L^{p,\infty}_\omega(\sigma)}
  :=\inf\Big\{\Norm{f-g}{L^{p,\infty}(\sigma)}: g\in L^{p,\infty}_\omega(\sigma)\Big\}.
\end{equation}
However, these norms also have an alternative formula; see Lemma \ref{lem:quotient-vs-limsup} below:
\begin{equation*}
  \Norm{[f]}{L^{p,\infty}(\sigma)/L^{p,\infty}_\omega(\sigma)}^p
  =\limsup_{\lambda\to\omega}\lambda^p \sigma(\abs{f}>\lambda).
\end{equation*}
Applying this formula with $\sigma\in\{\mu,\nu\}$ on both sides of \eqref{eq:quotient-bd}, we obtain \eqref{eq:limsup} with $\Theta(\omega)=1$. This completes the proof of the improved estimate for linear operators, and then also the proof of the whole theorem.
\end{proof}

In the previous proof, we used the following folklore identity, which we state and prove for completeness; essentially the same argument is found in \cite[Lemma 2.4]{FLSZ} for a slightly different version of the result.

\begin{lemma}\label{lem:quotient-vs-limsup}
For $f\in L^{p,\infty}(\mu)$, $\omega\in\{0,\infty\}$, and $L^{p,\infty}_\omega(\mu)$ as in \eqref{eq:weakLp-omega}, we have
\begin{equation}\label{eq:quotient-vs-limsup}
  \Norm{[f]}{L^{p,\infty}(\mu)/L^{p,\infty}_\omega(\mu)}^p
  =\limsup_{\lambda\to\omega}\lambda^p \mu(\abs{f}>\lambda).
\end{equation}
\end{lemma}

\begin{proof}
Recalling the low and high parts of a function from \eqref{eq:f-lowhigh}, it is easy to check that $f^{\kappa}\in L^{p,\infty}_0(\mu)$ and $f_{\kappa}\in L^{p,\infty}_\infty(\mu)$ for all $f\in L^{p,\infty}(\mu)$ and $\kappa\in(0,\infty)$. Hence, from the definition \eqref{eq:quotient-norm} of the quotient norm, we have
\begin{equation}\label{eq:quotient-up}
\begin{split}
  \Norm{[f]}{L^{p,\infty}(\mu)/L^{p,\infty}_\omega(\mu)}^p
  \leq\begin{cases}
     \Norm{f-f^\kappa}{L^{p,\infty}(\mu)}^p 
     =\Norm{f_\kappa}{L^{p,\infty}(\mu)}^p, & \text{if}\quad\omega=0, \\
     \Norm{f-f_\kappa}{L^{p,\infty}(\mu)}^p 
     =\Norm{f^\kappa}{L^{p,\infty}(\mu)}^p, & \text{if}\quad\omega=\infty, \end{cases}
\end{split}
\end{equation}
where
\begin{equation*}
\begin{split}
  \Norm{f_{\kappa}}{L^{p,\infty}(\mu)}^p
  &=\sup_{\lambda>0}\lambda^p\mu(\abs{f}\cdot 1_{\{\abs{f}\leq\kappa\}}>\lambda) 
  \leq\sup_{\lambda<\kappa}\lambda^p\mu(\abs{f}>\lambda), \\
  \Norm{f^{\kappa}}{L^{p,\infty}(\mu)}^p
  &=\sup_{\lambda>0}\lambda^p\mu(\abs{f}\cdot 1_{\{\abs{f}>\kappa\}}>\lambda) \\
  &=\sup_{\lambda>0}\lambda^p\mu(\abs{f}>\max(\kappa,\lambda))
  =\sup_{\lambda\geq\kappa}\lambda^p\mu(\abs{f}>\lambda).
\end{split}
\end{equation*}
Thus, taking $\lim_{\kappa\to\omega}$ on the right of \eqref{eq:quotient-up}, it follows that
\begin{equation}\label{eq:quotient-up2}
  \Norm{[f]}{L^{p,\infty}(\mu)/L^{p,\infty}_\omega(\mu)}^p
  \leq\limsup_{\lambda\to\omega}\lambda^p\mu(\abs{f}>\lambda).
\end{equation}

For the other direction, let $f\in L^{p,\infty}(\mu)$ and $g\in L^{p,\infty}_\omega(\mu)$. From
\begin{equation*}
\begin{split}
  \lambda^p\mu(\abs{f}>(1+\eps)\lambda)
  &\leq \lambda^p\mu(\abs{f-g}>\lambda)
    +\lambda^p\mu(\abs{g}>\eps\lambda) \\
 &\leq \Norm{f-g}{L^{p,\infty}(\mu)}^p+\lambda^p\mu(\abs{g}>\eps\lambda),
\end{split}
\end{equation*}
taking $\limsup_{\lambda\to\omega}$ of both sides, we conclude that
\begin{equation*}
\begin{split}
  (1+\eps)^{-p}\limsup_{\lambda\to\omega}\lambda^p\mu(\abs{f}>\lambda)
  &\leq\Norm{f-g}{L^{p,\infty}(\mu)}^p+\eps^{-p}\limsup_{\lambda\to\omega}
  \lambda^p\mu(\abs{g}>\lambda) \\
  &=\Norm{f-g}{L^{p,\infty}(\mu)}^p,\quad\text{since}\quad g\in L^{p,\infty}_\omega(\mu).
\end{split}
\end{equation*}
Taking the limit $\eps\to 0^+$ and then the infimum over all $g\in L^{p,\infty}_\omega(\mu)$, we obtain the converse of \eqref{eq:quotient-up2}. Together with \eqref{eq:quotient-up2}, this completes the proof of the lemma.
\end{proof}

\begin{example}\label{ex:Hardy}
We consider the space $(0,\infty)$ with measures $\mu=\nu=\ud x$ (the Lebesgue measure) and the Hardy operator $Tf(x):=x^{-1}\int_0^x f(y)\ud y$. This satisfies the assumptions of Theorem \ref{thm:liminf} with $q=1$ and $r=\infty$. For this operator, for all $p\in(1,\infty)$ and both $\omega\in\{0,\infty\}$, the estimate \eqref{eq:liminf} provided by Theorem \ref{thm:liminf} is the best possible among all estimates of this form, i.e., \eqref{eq:liminf} holds with $\theta=\theta_p:=1-\frac{1}{p}$ and it fails for all $\theta>\theta_p$.
\end{example}

\begin{proof}[Proof of Example \ref{ex:Hardy}]
From the pointwise bounds $\abs{Tf(x)}\leq\Norm{f}{\infty}$ and $\abs{Tf(x)}\leq x^{-1}\Norm{f}{1}$, and noting that $\Norm{x^{-1}}{L^{1,\infty}(0,\infty)}=1$, it is immediate that $T$ is bounded on $L^\infty(0,\infty)$, and from $L^1(0,\infty)$ to $L^{1,\infty}(0,\infty)$. In particular, the assumptions of Theorem \ref{thm:liminf} are satisfied with $q=1$ and $r=\infty$. 

We will consider a fixed $p\in(1,\infty)$ and ignore the dependence of multiplicative constants on $p$, as in Theorem \ref{thm:liminf}. On the other hand, we will pay attention to dependence on the small parameter $\delta\in(0,2^{-\frac{1}{p-1}})$. We consider the following functions $f_{p,\delta}$:
\begin{equation*}
  f_{p,\delta}(x):=
     \delta^{k},\quad\text{for}\quad x\in(\delta^{(1-k)p},\delta^{-k p}],\quad k\in\Z.
\end{equation*}
Then
\begin{equation*}
  \mu(\abs{f_{p,\delta}}>\lambda)
  =\abs{(0,\delta^{-kp}]}=\delta^{-kp}
  \quad\text{for}\quad\lambda\in[\delta^{(k+1)},\delta^{k}),\quad k\in\Z.
\end{equation*}
It follows that
\begin{equation}\label{eq:fpd}
\begin{split}
  \sup_{\lambda>0}\lambda\mu(\abs{f_{p,\delta}}>\lambda)^{\frac1p}
  &= \sup_{k\in\Z} \sup_{\lambda\in[\delta^{k+1},\delta^{k})}
  \lambda\cdot \delta^{-k} = 1, \\
  \liminf_{\lambda\to\omega}\lambda\mu(\abs{f_{p,\delta}}>\lambda)^{\frac1p}
  &=\liminf_{k\to\omega}\inf_{\lambda\in[\delta^{k+1},\delta^{k})}
     \lambda\cdot \delta^{-k}
  =\delta
\end{split}
\end{equation}
for both $\omega\in\{0,\infty\}$.
In particular, all functions $f_{p,\delta}$ have equal $L^{p,\infty}(\mu)$ norm $1$, while the related lower limits decrease to zero as $\delta\to 0$.

Let $g_{p,\delta}:=Tf_{p,\delta}$. For $x\in(\delta^{(1-k)p},\delta^{-kp}]$, we have
\begin{equation*}
  g_{p,\delta}(x)
  =\frac{1}{x}\Big(\sum_{j=-\infty}^{k-1}\delta^{j}\delta^{-jp}(1-\delta^p)
    +\delta^{k}(x-\delta^{(1-k)p})\Big).
\end{equation*}
For the sum, we have
\begin{equation*}
  \sum_{j=-\infty}^{k-1}\delta^{j}\delta^{-jp}(1-\delta^p)
  =\frac{\delta^{(k-1)(1-p)}}{1-\delta^{p-1}}(1-\delta^p)
  \sim\delta^{(k-1)(1-p)},
\end{equation*}
ignoring factors that are essentially $1$, since $\delta^p\leq\delta^{p-1}\leq\frac12$. Substituting back and replacing a sum by the maximum, it follows that
\begin{equation*}
\begin{split}
  g_{p,\delta}(x) 
  &\sim\frac{1}{x}\Big(\delta^{(k-1)(1-p)} +\delta^{k}(x-\delta^{(1-k)p})\Big) \\
  &\sim\begin{cases}
    x^{-1}\delta^{(k-1)(1-p)}, & x\in(\delta^{(1-k)p},\delta^{(1-k)p-1}], \\
    \delta^{k}, & x\in(\delta^{(1-k)p-1},\delta^{-kp}].     \end{cases}
\end{split}
\end{equation*}
Denoting by $h_{p,\delta}$ the right-hand of the previous display, we have
\begin{equation*}
  \mu(\abs{h_{p,\delta}}>\lambda)
    =\lambda^{-1}\delta^{(k-1)(1-p)},
    \quad\text{if}\quad \lambda\in[\delta^k,\delta^{k-1}),\quad k\in\Z.
\end{equation*}
Since $g_{p,\delta}\sim h_{p,\delta}$, we can estimate
\begin{equation*}
\begin{split}
  \liminf_{\lambda\to\omega}\lambda\cdot\mu(\abs{g_{p,\beta}}>\lambda)^{\frac1p}
  &\sim\liminf_{\lambda\to\omega}\lambda\cdot\mu(\abs{h_{p,\beta}}>\lambda)^{\frac1p} \\
  &=\liminf_{k\to 1/\omega}\inf_{\lambda\in[\delta^{k},\delta^{k-1})}
  \lambda\cdot\lambda^{-\frac1p}\delta^{(k-1)(\frac1p-1)} \\
  &=\liminf_{k\to1/\omega}
     \delta^{k(1-\frac1p)}\delta^{(k-1)(\frac1p-1)}
     =\delta^{1-\frac1p},
\end{split}
\end{equation*}
where $1/0:=\infty$ and $1/\infty:=0$.
A comparison with \eqref{eq:fpd} as $\delta\to 0$ shows that, if \eqref{eq:liminf} holds for the Hardy operator $T$ on $(0,\infty)$ with $\mu=\nu=\ud x$, for some $\theta\in[0,1]$ and all $f\in L^{p,\infty}(0,\infty)$, then necessarily $\theta\leq\theta_p:= 1-\frac{1}{p}$. On the other hand, Theorem \ref{thm:liminf} with $q=1$ and $r=\infty$ shows that we have this estimate with $\theta=\theta_p$. Thus, for the Hardy operator $T$ on $(0,\infty)$ with $\mu=\nu=\ud x$, the estimate \eqref{eq:liminf} provided by Theorem \ref{thm:liminf} is the best possible, for both $\omega\in\{0,\infty\}$.
\end{proof}

We will also need the following variant of Theorem \ref{thm:liminf}, where the domain $L^p$ spaces of the operator are replaced by Schatten classes.

\begin{corollary}\label{cor:liminf}
Let $0<q<p<r\leq\infty$ and 
\begin{equation*}
  T:S^{q}(\mathcal H)+S^{r}(\mathcal H)\to L^{q,\infty}(\nu)+L^{r,\infty}(\nu)
\end{equation*}
be a quasi-sublinear operator (in the sense of \eqref{eq:quasi-sub})
with bounded restrictions
\begin{equation*}
  T:S^{q}(\mathcal H)\to L^{q,\infty}(\nu),\quad
  T:S^{r}(\mathcal H)\to L^{r,\infty}(\nu),
\end{equation*}
where $L^{\infty,\infty}:=L^\infty$ and $S^\infty(\mathcal H):=\mathcal K(\mathcal H)$ is the space of compact operators on the Hilbert space $\mathcal H$.
If $A\in S^{p,\infty}(\mathcal H)$, then
\begin{subequations}\label{eq:liminfsup-Sp}
\begin{align}
  \liminf_{\lambda\to 0}
    \lambda\cdot\nu(\abs{T(A)}>\lambda)^{\frac1p}
    &\lesssim \Big(\liminf_{n\to\infty} a_n(A)\cdot n^{\frac1p}\Big)^\theta \Norm{A}{S^{p,\infty}(\mathcal H)}^{1-\theta},
    \label{eq:liminf-Sp} \\
    \limsup_{\lambda\to 0}
    \lambda\cdot\nu(\abs{T(A)}>\lambda)^{\frac1p}
    &\lesssim \Big(\limsup_{n\to\infty}a_n(A)\cdot n^{\frac1p}\Big)^{\Theta} 
    \Norm{A}{S^{p,\infty}(\mathcal H)}^{1-\Theta},
    \label{eq:limsup-Sp} \\
    \lim_{\lambda\to\infty}
    \lambda\cdot\nu(\abs{T(A)}>\lambda)^{\frac1p}
    &=0.\label{eq:lim-triv}
\end{align}
\end{subequations}
where $\theta$ and $\Theta:=\Theta(0)$ are as in Theorem \ref{thm:liminf}.
\end{corollary}

Only \eqref{eq:liminf-Sp} will feature in our main application, but \eqref{eq:limsup-Sp} and \eqref{eq:lim-triv} are obtained as easy by-products and recorded for completeness.

\begin{proof}
For each fixed pair of orthonormal systems $(e_n)_{n=1}^\infty$ and $(f_n)_{n=1}^\infty$ in $\mathcal H$, we consider the operator
\begin{equation*}
  \mathcal E: c_0\to\mathcal K(\mathcal H),\quad
  \sigma=(\sigma_n)_{n=1}^\infty\mapsto \sum_{n=1}^\infty\sigma_n e_n\otimes f_n.
\end{equation*}
Then $\Norm{\mathcal E\sigma}{S^{u,v}(\mathcal H)}=\Norm{\sigma}{\ell^{u,v}}$ for all $u,v$. Hence, if $T:S^q(\mathcal H)+S^r(\mathcal H)\to L^{q,\infty}(\nu)+L^{r,\infty}(\nu)$ satisfies the assumptions of Corollary \ref{cor:liminf}, then $T\circ\mathcal E:\ell^q+\ell^r\to L^{q,\infty}(\nu)+L^{r,\infty}(\nu)$ satisfies the assumptions of Theorem \ref{thm:liminf} verbatim, noting that $\ell^{u,v}=L^{u,v}(\mu)$, when $\mu$ is the counting measure of $\Z_{>0}$. Since $\mathcal E$ is linear, we see that $T\circ\mathcal E$ is quasi-sublinear if $T$ is, and linear if $T$ is. An application of Theorem \ref{thm:liminf} then shows that $T\circ\mathcal E$ in place of $T$ satisfies all the conclusions of Theorem \ref{thm:liminf}, i.e., 
\begin{subequations}\label{eq:com-to-nc}
\begin{align}
  \liminf_{\lambda\to\omega}\lambda\cdot\nu(\abs{T(\mathcal E\sigma)}>\lambda)^{\frac1p}
  &\lesssim
  \Big(\liminf_{\lambda\to\omega}\lambda\cdot\mu(\abs{\sigma}>\lambda)^{\frac1p}\Big)^\theta
  \Norm{\sigma}{L^{p,\infty}(\mu)}^{1-\theta}, \\
  \limsup_{\lambda\to\omega}\lambda\cdot\nu(\abs{T(\mathcal E\sigma)}>\lambda)^{\frac1p}
  &\lesssim
  \Big(\limsup_{\lambda\to\omega}\lambda\cdot\mu(\abs{\sigma}>\lambda)^{\frac1p}\Big)^\Theta
  \Norm{\sigma}{L^{p,\infty}(\mu)}^{1-\Theta(\omega)}.
\end{align}
\end{subequations}
Given $A\in S^{p,\infty}(\mathcal H)$, we write down its singular value decomposition
\begin{equation*}
  A=\sum_{n=1}^\infty a_n(A)\cdot e_n(A)\otimes f_n(A),
\end{equation*}
where $(a_n(A))_{n=1}^\infty$ are the singular values and $(e_n(A))_{n=1}^\infty$ and $(f_n(A))_{n=1}^\infty$ are some orthonormal systems.

Noting that \eqref{eq:com-to-nc} holds uniformly over all choices of $\mathcal E$, we apply it with the specific choice $e_n=e_n(A)$ and $f_n=f_n(A)$. We also take $\sigma=(a_n(A))_{n=1}^\infty$. With these choices, it follows that
\begin{equation}\label{eq:Esigma}
  \mathcal E\sigma=A,\qquad\Norm{\sigma}{L^{p,\infty}(\mu)}=\Norm{\sigma}{\ell^{p,\infty}}
  =\Norm{A}{S^{p,\infty}(H)},
\end{equation}
and
\begin{equation*}
  \mu(\abs{\sigma}>\lambda)
  =\#\{n:a_n(A)>\lambda\}.
\end{equation*}
We note that the sequence $\sigma=(a_n(A))_{n=1}^\infty\in\ell^{p,\infty}$ is in particular bounded. Hence $\mu(\abs{\sigma}>\lambda)=0$ for all $\lambda\geq\Norm{\sigma}{\infty}=a_1(A)$, which shows the vanishing of the right, and hence the left, sides of \eqref{eq:com-to-nc} for $\omega=\infty$. This proves \eqref{eq:lim-triv}.

We will then concentrate on the lower and upper limits as $\lambda\to 0$.
Let $(n_j)_{j=1}^\infty$ be the indices such that $n_1=1$ and, for each $j\geq 2$,
\begin{equation*}
  a_{n_{j-1}}(A)=\ldots=a_{n_{j}-1}(A)>a_{n_{j}}(A).
\end{equation*}
Then
\begin{equation}\label{eq:liminf-lan}
\begin{split}
  \liminf_{\lambda\to 0} \lambda &\cdot \mu(\abs{\sigma}>\lambda)^{\frac1p} \\
  &=\liminf_{j\to\infty}\inf_{\lambda\in[a_{n_{j+1}(A)},a_{n_j}(A))}
    \lambda
    (\#\{n:a_n(A)>\lambda\})^{\frac1p} \\
  &=\liminf_{j\to\infty}\inf_{\lambda\in[a_{n_{j}(A)},a_{n_{j-1}}(A))}
    \lambda\cdot (n_j-1)^{\frac1p} \\
  &=\liminf_{j\to\infty}a_{n_{j}}(A)\cdot (n_j-1)^{\frac1p} 
  =\liminf_{j\to\infty}a_{n_{j}}(A)\cdot n_j^{\frac1p} \\  
  &=\liminf_{j\to\infty}\inf_{n\in[n_j,n_{j+1})} a_{n_{j}}(A)\cdot n_j^{\frac1p} 
  =\liminf_{n\to\infty} a_n(A)\cdot n^{\frac1p}.
\end{split}
\end{equation}
Substituting \eqref{eq:Esigma} and \eqref{eq:liminf-lan} into \eqref{eq:com-to-nc}, we obtain \eqref{eq:liminf-Sp}.
Similarly, we have
\begin{equation}\label{eq:limsup-lan}
\begin{split}
  \limsup_{\lambda\to 0} \lambda &\cdot \mu(\abs{\sigma}>\lambda)^{\frac1p} \\
  &=\limsup_{j\to\infty}\sup_{\lambda\in[a_{n_{j+1}(A)},a_{n_j}(A))}
    \lambda
    (\#\{n:a_n(A)>\lambda\})^{\frac1p} \\
  &=\limsup_{j\to\infty}\sup_{\lambda\in[a_{n_{j}(A)},a_{n_{j-1}}(A))}
    \lambda\cdot (n_j-1)^{\frac1p} \\
  &=\limsup_{j\to\infty}a_{n_{j-1}}(A)\cdot (n_j-1)^{\frac1p} \\
  &=\limsup_{j\to\infty}\sup_{n\in[n_{j-1},n_j)} a_{n}(A)\cdot n^{\frac1p} 
  =\limsup_{n\to\infty} a_n(A)\cdot n^{\frac1p}.
\end{split}
\end{equation}
Substituting \eqref{eq:Esigma} and \eqref{eq:limsup-lan} into \eqref{eq:com-to-nc}, we obtain \eqref{eq:limsup-Sp}. This completes the proof.
\end{proof}

\section{Hidden operators in commutator lower bounds}\label{sec:implicit}

In order to apply Theorem \ref{thm:liminf} to commutator lower bounds, we need to reformulate some estimates obtained in \cite{Hyt:Osc,Hyt:Sp,HW:frac} as the boundedness of appropriate operators. The section at hand is dedicated to this task.

\begin{proposition}\label{prop:osc<Phi}
Let $b\in L^1_{\loc}(\mu)$ and $T$ be an operator with strongly non-degenerate singular kernel. Then there is a linear operator $\Phi:\bddlin(L^2(\mu))\to\ell^\infty(\mathscr D)$ such that
\begin{equation}\label{eq:osc<Phi}
   \operatorname{osc}_\mu(b,B_Q)
   :=\inf_{c\in\C}\fint_B\abs{b-c}\ud\mu
   \lesssim\Phi([b,T])(Q)
\end{equation}
and $\Phi:S^{p,q}(L^2(\mu))\to \ell^{p,q}(\mathscr D)$ is bounded for all $p\in(1,\infty)$ and $q\in[1,\infty]$. The operator $\Phi$ is allowed to depend on both $b$ and $T$, but the implicit constants in \eqref{eq:osc<Phi} and the boundedness of $\Phi$ are independent of $b$.
\end{proposition}

\begin{proof}
Under the assumptions on $b$ and $T$, we first claim the following: for every ball $B\subset X$, there are subsets $E,F\subset B^*:=C\cdot B$ such that
\begin{equation}\label{eq:osc<comm-base}
  \operatorname{osc}_\mu(b,B)\lesssim\frac{\abs{\pair{[b,T]1_E}{1_F}}}{\mu(B)}.
\end{equation}
Indeed, for $\phi$-fractional integrals $T$, this is \cite[Corollary 6.4]{HW:frac}, and the case of singular integrals follows by the same argument, taking $\phi\equiv 1$.

We then apply \eqref{eq:osc<comm-base} with each $B=B_Q$ to find subsets $E_Q,F_Q\subset B_Q^*:=C\cdot B_Q$ such that
\begin{equation*}
   \operatorname{osc}_\mu(b,B_Q)
   \lesssim\frac{\abs{\pair{[b,T]1_{E_Q}}{1_{F_Q}}}}{\mu(B_Q)}
   =:\pair{[b,T]e_Q}{h_Q}
\end{equation*}
provided that
\begin{equation*}
  e_Q:=\frac{1_{E_Q}}{\mu(B_Q)^{\frac12}},\qquad
  h_Q:=\sigma_Q\frac{1_{F_Q}}{\mu(B_Q)^{\frac12}},\qquad
  \sigma_Q:=\overline{\sign(\pair{[b,T]1_{E_Q}}{1_{F_Q}})},
\end{equation*}
where $\sign(z):=z/\abs{z}$ for $z\in\C\setminus\{0\}$. Hence, the linear operator
\begin{equation*}
  \Phi(A):=\pair{Ae_Q}{h_Q}
\end{equation*}
satisfies \eqref{eq:osc<Phi}, and it remains to check the claimed boundedness properties. Evidently $\Norm{e_Q}{L^2(\mu)},\Norm{h_Q}{L^2(\mu)}\lesssim 1$, so indeed $\Phi:\bddlin(L^2(\mu))\to\ell^\infty(\mathscr D)$ is bounded.

By \cite[Corollary 7.7]{Hyt:Sp}, we have
\begin{equation*}
  \Norm{\Phi}{S^{p,q}(L^2(\mu))\to \ell^{p,q}(\mathscr D)}
  \lesssim\Norm{M_{\mathcal E} }{L^2(\mu)\times L^2(\mu)\to L^1(\mu)},
\end{equation*}
where $M_{\mathcal E}$ is the bi-sublinear maximal operator
\begin{equation*}
  M_{\mathcal E}(f,g)(x)
  :=\sup_{Q\in\mathscr D}1_Q(x)\frac{\abs{\pair{f}{e_Q}\pair{g}{h_Q}}}{\mu(Q)}.
\end{equation*}
But it is immediate from the definition of $e_Q$ and $h_Q$ that
\begin{equation*}
  M_{\mathcal E}(f,g)(x)
  \lesssim \sup_{Q\in\mathscr D}1_Q(x)\fint_{B_Q^*}\abs{f}\ud\mu\cdot\fint_{B_Q^*}\abs{g}\ud\mu
  \leq Mf(x)Mg(x),
\end{equation*}
where $M$ is the Hardy--Littlewood maximal operator. The $L^2(\mu)\times L^2(\mu)\to L^1(\mu)$ bound of $M_{\mathcal E}$ then follows from the Cauchy--Schwarz inequality and the Hardy--Littlewood maximal theorem on the $L^2(\mu)$ boundedness of $M$.
\end{proof}

\begin{proposition}\label{prop:Ainfty}
Let $\mu,\nu$ be two doubling measures on a metric space $(X,\rho)$ such that they satisfy the $A_\infty$ relation \eqref{eq:Ainfty}. Then there is a sublinear operator $S_{\nu,\mu}$ on $c_0(\mathscr D)$ such that
\begin{equation*}
  \operatorname{osc}_\nu(b,B_Q)
  \lesssim S_{\mu,\nu}(\{\operatorname{osc}_\mu(b,B_R)\}_{R\in\mathscr D})(Q)
\end{equation*}
and $S_{\mu,\nu}$ is bounded on $\ell^{p,q}(\mathscr D)$ for all $p\in(0,\infty)$ and $q\in(0,\infty]$.
\end{proposition}

\begin{proof}
Under the assumptions of the proposition, the proof of \cite[Proposition 1.2]{Hyt:Osc} contains the estimate
\begin{equation}\label{eq:osc-nu<osc-mu-t}
  \operatorname{osc}_\nu(b,B)
  \lesssim\operatorname{osc}_{\mu,t}(b,B)
  :=\inf_{c\in\C}\Big(\fint_B\abs{b-c}^t\ud\mu\Big)^{\frac1t}
\end{equation}
for every metric ball $B$ and some $t>1$ depending on the pair of measures $(\mu,\nu)$. In particular, this holds for each $B=B_Q$.

Next, let $\mathscr D^m$, where $m=1,\ldots,M$, be the boundedly many {\em adjacent dyadic systems} with the property
\begin{equation}\label{eq:adjacent}
  \forall B=B(x,r)\ \exists\ R\in\mathscr D^*:=\bigcup_{m=1}^M\mathscr D^m:\quad
  B\subset R,\quad\operatorname{diam}(R)\lesssim r;
\end{equation}
these are constructed in \cite[Theorem 4.1]{HK:12}. We define $\rho:\mathscr D\to\mathscr D^*$ by letting $\rho(Q)$ be a choice of a cube $R\in\mathscr D^*$ guaranteed by the defining property \eqref{eq:adjacent} applied to $B=B_Q$.

Given $R\in\mathscr D^*$, we claim that there are only boundedly many cubes $Q\in\mathscr D$ with $\rho(Q)=R$. Indeed, all such $Q$ satisfy $Q\subset R$, $\operatorname{diam}(Q)\sim\operatorname{diam}(R)$, and hence $\mu(Q)\sim\mu(R)$. Thus, these cubes $Q$ belong to at most boundedly many generations. And there are at most boundedly many $Q$ of a fixed generation, hence only boundedly many $Q$ altogether.

We have now seen that
\begin{equation}\label{eq:osc-BQ<osc-R}
   \operatorname{osc}_{\mu,t}(b,B_Q)
   \lesssim\operatorname{osc}_{\mu,t}(b,\rho(Q))
   =:C_\rho(\{\operatorname{osc}_{\mu,t}(b,R)\}_{R\in\mathscr D^*})(Q),\quad Q\in\mathscr D,
\end{equation}
where $C_\rho$ is the composition operator $C_\rho F:=F\circ \rho$. It is evident that $C_\rho:\ell^\infty(\mathscr D^*)\to\ell^\infty(\mathscr D)$ is linear. Moreover, for all $p\in(0,\infty)$,
\begin{equation*}
  \Norm{C_\rho F}{\ell^p(\mathscr D)}^p
  =\sum_{Q\in\mathscr D} F(\rho(Q))^p
  =\sum_{R\in\mathscr D^*}\sum_{\substack{Q\in\mathscr D \\ \rho(Q)=R}} F(R)^p 
  \lesssim \sum_{R\in\mathscr D^*} F(R)^p 
  =\Norm{F}{\ell^p(\mathscr D^*)}^p,
\end{equation*}
since $\#\{Q\in\mathscr D: \rho(Q)=R\}\lesssim 1$, as we just argued. By interpolation, it also follows that
\begin{equation*}
  \Norm{C_\rho F}{\ell^{p,q}(\mathscr D)}
  \lesssim \Norm{F}{\ell^{p,q}(\mathscr D^*)}
\end{equation*}
for all $p\in(0,\infty)$ and $q\in(0,\infty]$.

Next, following \cite[page 34]{Hyt:Sp}, we need the notion of the Carleson operator with parameter $r\in[1,\infty)$,
\begin{equation*}
  (\operatorname{Car}_{\mu,r;\mathscr D^m} \lambda)(P):=
  \Big(\frac{1}{\mu(P)}\sum_{\substack{ Q\in\mathscr D^m \\ Q\subseteq P}}
    \mu(Q)\abs{\lambda_Q}^r\Big)^{\frac1r},\quad
    P\in\mathscr D^m,\quad\lambda=(\lambda_Q)_{Q\in\mathscr D^m}.
\end{equation*}
These are sublinear operators bounded on $\ell^{p,q}(\mathscr D^m)$ for all $p\in(0,\infty)$ and $q\in(0,\infty]$. Indeed, case $r=1$ is \cite[Proposition 7.4]{Hyt:Sp}, and the general case follows from this via
\begin{equation*}
\begin{split}
  \Norm{\operatorname{Car}_{\mu,r;\mathscr D^m} \lambda}{\ell^{p,q}}
  =\Norm{(\operatorname{Car}_{\mu,r;\mathscr D^m}(\abs{\lambda}^r))^{\frac1r}}{\ell^{p,q}} 
  &=\Norm{\operatorname{Car}_{\mu,r;\mathscr D^m}(\abs{\lambda}^r)}{\ell^{\frac pr,\frac qr}}^{\frac1r} \\
  &\lesssim\Norm{\abs{\lambda}^r}{\ell^{\frac pr,\frac qr}}^{\frac1r}
  =\Norm{\lambda}{\ell^{p,q}}
\end{split}
\end{equation*}
We also define the natural extension
\begin{equation*}
  (\operatorname{Car}_{\mu,r;\mathscr D^*} \lambda)(P):=
  \sum_{m=1}^M 1_{\mathscr D^m}(P)(\operatorname{Car}_{\mu,r;\mathscr D^m} \lambda)(P),
  \quad P\in\mathscr D^*,\quad\lambda=(\lambda_Q)_{Q\in\mathscr D^*},
\end{equation*}
and it is evident that these inherit the same boundedness on $\ell^{p,q}(\mathscr D^*)$ for $p\in(0,\infty)$ and $q\in(0,\infty]$.

Next, the proof of \cite[Proposition 11.3]{Hyt:Sp}, see \cite[page 50]{Hyt:Sp}, contains the estimate
\begin{equation*}
  \operatorname{osc}_{\mu,t}(b,R)
  \lesssim
  \operatorname{Car}_{\mu,t;\mathscr D^m}(
  \{\operatorname{osc}_{\mu}(b,P)\}_{P\in\mathscr D^m})(R),\quad
  R\in\mathscr D^m,
\end{equation*}
which trivially implies
\begin{equation}\label{eq:osc-t<Car}
  \operatorname{osc}_{\mu,t}(b,R)
  \lesssim
  \operatorname{Car}_{\mu,t;\mathscr D^*}(
  \{\operatorname{osc}_{\mu}(b,P)\}_{P\in\mathscr D^*})(R),\quad
  R\in\mathscr D^*.
\end{equation}

Finally, for every $P\in\mathscr D^*$, let $\sigma(P)=S\in\mathscr D$ be the unique cube of the same generation with centre $z_S\in P$. If the constant $c$ in the definition of $B_S$ is large enough, it follows that $P\subset B_S$, and hence
\begin{equation}\label{eq:osc-P<osc-BQ}
   \operatorname{osc}_\mu(b,P)\lesssim\operatorname{osc}_\mu(b,B_{\sigma(P)})
   =C_{\sigma}(\{\operatorname{osc}_\mu(b,B_S)\}_{S\in\mathscr D})(P),
\end{equation}
where $C_\sigma$ is the composition operator induced by $\sigma:\mathscr D^*\to\mathscr D$ in the same way as $C_\rho$ was induced by $\sigma:\mathscr D\to\mathscr D^*$. In a similar way, one checks that $\#\{P\in\mathscr D^*:\sigma(P)=S\}$ is uniformly bounded over all $S\in\mathscr D$, and hence $C_\sigma:\ell^{p,q}(\mathscr D)\to\ell^{p,q}(\mathscr D^*)$ is bounded for all $p\in(0,\infty)$ and $q\in(0,\infty]$.

Combining the estimates above, we obtain
\begin{equation*}
\begin{split}
  \operatorname{osc}_{\nu}(b;B_Q)
  &\lesssim \operatorname{osc}_{\mu,t}(b,B_Q)\qquad\text{by \eqref{eq:osc-nu<osc-mu-t}} \\
  &\lesssim C_\rho(\{\operatorname{osc}_{\mu,t}(b,R)\}_{R\in\mathscr D^*})(Q)
    \qquad\text{by \eqref{eq:osc-BQ<osc-R}} \\
  &\lesssim C_\rho\circ \operatorname{Car}_{\mu,t;\mathscr D^*}(
  \{\operatorname{osc}_{\mu}(b,P)\}_{P\in\mathscr D^*})(Q)
    \qquad\text{by \eqref{eq:osc-t<Car}} \\
  &\lesssim C_\rho\circ \operatorname{Car}_{\mu,t;\mathscr D^*}\circ\, C_\sigma(
  \{\operatorname{osc}_{\mu}(b,B_S)\}_{S\in\mathscr D})(Q)
    \qquad\text{by \eqref{eq:osc-P<osc-BQ}},
\end{split}
\end{equation*}
where
\begin{equation*}
\begin{split}
  C_\sigma &:\ell^{p,q}(\mathscr D)\to\ell^{p,q}(\mathscr D^*),\\
  \operatorname{Car}_{\mu,t;\mathscr D^*} &:\ell^{p,q}(\mathscr D^*)\to\ell^{p,q}(\mathscr D^*),\\
  C_\rho &:\ell^{p,q}(\mathscr D^*)\to\ell^{p,q}(\mathscr D)
\end{split}
\end{equation*}
are all bounded for all $p\in(0,\infty)$ and $q\in(0,\infty]$. The operators $C_\sigma$ and $C_\rho$ are linear, while
 $\operatorname{Car}_{\mu,t;\mathscr D^*}$ is sublinear; thus, the composition
 \begin{equation*}
  S_{\mu,\nu}:=C_\rho\circ \operatorname{Car}_{\mu,t;\mathscr D^*}\circ\, C_\sigma
\end{equation*}
is a sublinear operator bounded on $\ell^{p,q}(\mathscr D)$ for all $p\in(0,\infty)$ and $q\in(0,\infty]$. This completes the proof.
\end{proof}

\begin{corollary}\label{cor:osc-nu<Phi}
Let $\mu,\nu$ be two doubling measures on a metric space $(X,\rho)$ that satisfy the $A_\infty$ relation \eqref{eq:Ainfty}. Let $b\in L^1_{\loc}(\mu)$ and $T\in\bddlin(L^2(\mu))$ be an operator with strongly non-degenerate singular kernel. Then there is a sublinear operator $\tilde\Phi:\bddlin(L^2(\mu))\to\ell^\infty(\mathscr D)$ such that
\begin{equation}\label{eq:osc-nu<Phi}
   \operatorname{osc}_\nu(b,B_Q)
   \lesssim\tilde\Phi([b,T])(Q)
\end{equation}
and $\tilde\Phi:S^{p,q}(L^2(\mu))\to \ell^{p,q}(\mathscr D)$ is bounded for all $p\in(1,\infty)$ and $q\in[1,\infty]$. The operator $\tilde\Phi$ is allowed to depend on both $b$ and $T$, but the implicit constants in \eqref{eq:osc-nu<Phi} and the boundedness of $\Phi$ are independent of $b$.
\end{corollary}

\begin{proof}
It suffices to take $\tilde\Phi:=S_{\mu,\nu}\circ\Phi$, where $\Phi:S^{p,q}(L^2(\mu))\to\ell^{p,q}(\mathscr D)$ is the linear operator from Proposition \ref{prop:osc<Phi}, and $S_{\mu,\nu}:\ell^{p,q}(\mathscr D)\to\ell^{p,q}(\mathscr D)$ is the sublinear operator from Proposition \ref{prop:Ainfty}.
\end{proof}

\begin{corollary}\label{cor:osc<comm-asymp}
Let $\mu,\nu$ be two doubling measures on a metric space $(X,\rho)$ that satisfy the $A_\infty$ relation \eqref{eq:Ainfty}. Let $b\in L^1_{\loc}(\mu)$ and $T\in\bddlin(L^2(\mu))$ be an operator with strongly non-degenerate singular kernel.
Then, for every $p\in(1,\infty)$ and $q\in[1,\infty]$,
\begin{equation}\label{eq:osc<comm-bound}
  \Norm{\{\operatorname{osc}_\nu(b,B_Q)\}_{Q\in\mathscr D}}{\ell^{p,q}(\mathscr D)}
  \lesssim \Norm{[b,T]}{S^{p,q}(L^2(\mu))},
\end{equation}
and there is $\theta\in(0,1)$ depending only on $p$ such that
\begin{equation}\label{eq:osc<comm-asymp}
\begin{split}
  \liminf_{\lambda\to 0}\lambda \cdot
  (\#\{Q\in\mathscr D &: \operatorname{osc}_\nu(b,B_Q)>\lambda\})^{\frac 1p} \\
  &\lesssim
  \Big(\liminf_{n\to\infty} a_n([b,T])\cdot n^{\frac1p}\Big)^\theta
  \Norm{[b,T]}{S^{p,\infty}(L^2(\mu))}^{1-\theta}.
\end{split}
\end{equation}
In fact, this holds for every $\theta\in(0,1-\frac1p)$.
\end{corollary}

\begin{proof}
The norm bound \eqref{eq:osc<comm-bound} is an immediate consequence of Corollary \ref{cor:osc-nu<Phi} (or, under slightly stronger assumptions on $T$, the combination of \cite[Proposition 1.2]{Hyt:Osc} and \cite[Proposition 1.26]{Hyt:Sp}). Moreover, by Corollary \ref{cor:osc-nu<Phi}, the sublinear operator $\tilde\Phi$ in \eqref{eq:osc-nu<Phi} satisfies the assumptions, and hence the conclusions, of Corollary \ref{cor:liminf} for all $1<q<p<r<\infty$.
Thus \eqref{eq:osc<comm-asymp} holds with $\theta=\frac{(p-q)(r-p)}{p(r-q)}$. Taking $q\to 1^+$ and $r\to\infty$, we can reach any $\theta\in(0,1-\frac1p)$.
\end{proof}

\section{Oscillations vs. Sobolev norms}\label{sec:Osc}

In this section, we compare the asymptotics of the discrete oscillation norms of Section \ref{sec:implicit} with their continuous counterparts, for which we can quote a relation to Sobolev norms from \cite{HK:W1p}. This will then allow us to complete the proof of case \eqref{it:main-reg} of Theorem \ref{thm:main}.

\begin{proposition}\label{prop:cont-discrete-osc}
Let $(X,\rho,\nu)$ be Ahlfors $d$-regular and
\begin{equation*}
  \ud\nu_d(x,t):=\ud\nu(x)\frac{\ud t}{t^{d+1}}.
\end{equation*}
Then, for all $b\in L^1_{\loc}(\nu)$,
we have
\begin{equation}\label{eq:disc-cont-bound}
  \Norm{\{\operatorname{osc}_{1,\nu}(b,B_Q)\}_{Q\in\mathscr D}}{\ell^{d,\infty}(\mathscr D)}
  \sim \Norm{\operatorname{osc}_{1,\nu}(b,B(\cdot,\cdot))}{L^{d,\infty}(\nu_d)},
\end{equation}
and
\begin{equation}\label{eq:disc-cont-asymp}
\begin{split}
  \liminf_{\lambda\to 0}\lambda &\cdot(\#\{Q\in\mathscr D 
  :\operatorname{osc}_{1,\nu}(b,B_Q)>\lambda\})^{\frac1d} \\
  &\sim
  \liminf_{\lambda\to 0}\lambda\cdot\nu_d(\{(x,t)\in X\times(0,\infty):
  \operatorname{osc}_{1,\nu}(b,B(x,t))>\lambda\})^{\frac1d}.
\end{split}
\end{equation}
A bound like \eqref{eq:disc-cont-asymp} also holds with both $\liminf$'s replaced by $\limsup$'s.
\end{proposition}

\begin{proof}
The norm bound \eqref{eq:disc-cont-bound} is simply a restatement of \cite[Proposition 12.1]{Hyt:Sp}. The asymptotic bound \eqref{eq:disc-cont-asymp} is also essentially contained in the proof of \cite[Proposition 12.1]{Hyt:Sp}, as we will next indicate.

By the first displays in \cite[pages 55 and 56]{Hyt:Sp}, respectively, we have
\begin{equation*}
   \#\{Q\in\mathscr D:\operatorname{osc}_{\nu}(b,B_Q)>\kappa\}
   \lesssim\nu_d(\{(x,t):\operatorname{osc}_{\nu}(b,B(x,t))>c\kappa\}).
\end{equation*}
and
\begin{equation*}
  \nu_d(\{(x,t):\operatorname{osc}_{\nu}(b,B(x,t))>\kappa\})
   \lesssim\#\{Q\in\mathscr D:\operatorname{osc}_{\nu}(b,B_Q)>c\kappa\}.
\end{equation*}
Taking $d$th roots, multiplying by $\lambda$, and taking $\liminf_{\lambda\to 0}$, it follows that
\begin{equation*}
\begin{split}
  \liminf_{\lambda\to 0}\lambda &(\#\{Q\in\mathscr D:\operatorname{osc}_{\nu}(b,B_Q)>\lambda\})^{\frac1d} \\
  &\sim
  \liminf_{\lambda\to 0}\lambda\cdot\nu_d(\{(x,t):\operatorname{osc}_{\nu}(b,B(x,t))>\lambda\})^{\frac1d}.
\end{split}
\end{equation*}
The proof of the limsup version is exactly the same, simply taking $\limsup_{\lambda\to 0}$ instead of $\liminf_{\lambda\to 0}$ in the last step.
\end{proof}

\begin{theorem}[\cite{HK:W1p}]\label{thm:HK-main}
Let $1\leq q<p<\infty$ and suppose that $(X,\rho,\nu)$ satisfies the $(1,q)$-Poincar\'e inequality. Let $f\in L^1_{\loc}(\nu)$. Then
\begin{enumerate}[\rm(1)]
  \item\label{it:M1p-weakLp} $f\in\dot M^{1,p}(\nu)$ if and only if $\operatorname{osc}_{1,\nu}(f,B(\cdot,\cdot))\in L^{p,\infty}(\nu_p)$, and
\begin{equation}\label{eq:HK-bound}
  \Norm{f}{\dot M^{1,p}(\nu)}
  \sim\Norm{\operatorname{osc}_{\nu}(f,B(\cdot,\cdot))}{L^{p,\infty}(\nu_p)}.
\end{equation}
  \item\label{it:HK-asymp} Under the equivalent conditions of part \eqref{it:M1p-weakLp}, we also have
\begin{equation}\label{eq:HK-asymp}
    \Norm{f}{\dot M^{1,p}(\nu)}
    \sim\liminf_{\lambda\to 0}\lambda\cdot\nu_p(\{\operatorname{osc}_{\nu}(f,B(\cdot,\cdot))>\lambda\})^{\frac 1p}.
\end{equation}
\end{enumerate}
\end{theorem}

\begin{proof}
This is a version of \cite[Theorem 1.1]{HK:W1p} noted in \cite[Remark 2.9]{HK:W1p}: the theorem is stated assuming that $(X,\rho,\nu)$ is a complete metric space with $(1,p)$-Poincar\'e inequality, which implies $(1,q)$-Poincar\'e inequality for some $q\in[1,p)$ by \cite{KZ:08}. In the said remark, it is noted that completeness is only used via this implication, and this assumption could be replaced by assuming a $(1,q)$-Poincar\'e inequality.
\end{proof}

\begin{corollary}\label{cor:Sob<comm}
Let $\mu,\nu$ be two doubling measures on a metric space $(X,\rho)$ that satisfy the $A_\infty$ relation \eqref{eq:Ainfty}.
Let $(X,\rho,\nu)$ be Ahlfors $d$-regular for some $d\in(1,\infty)$ and satisfy the $q$-Poincar\'e inequality for some $q\in[1,d)$.
Let $b\in L^1_{\loc}(\mu)$ and $T$ be a non-degenerate $\mu$-singular integral.
If $[b,T]\in S^{d,\infty}(L^2(\mu))$, then
\begin{equation}\label{eq:Sob<comm-bound}
    \Norm{ b }{ \dot M^{1,d}(\nu)}
    \lesssim\Norm{[b,T]}{S^{d,\infty}(L^2(\mu))}.
\end{equation}
In this case, there is $\theta\in(0,1)$ depending only on $d$ such that
\begin{equation}\label{eq:Sob<comm-asymp}
    \Norm{ b }{ \dot M^{1,d}(\nu)}
    \lesssim
    \Big(\liminf_{n\to\infty}a_n([b,T])\cdot n^{\frac1d}\Big)^\theta
    \Norm{[b,T]}{S^{d,\infty}(L^2(\mu))}^{1-\theta}
\end{equation}
In fact, this holds for every $\theta\in(0,1-\frac1d)$.
\end{corollary}

\begin{proof}
First, by \eqref{eq:osc<comm-bound}, we have
\begin{equation}\label{eq:by-osc<comm}
  \Norm{\{\operatorname{osc}_\nu(f,B_Q)\}_{Q\in\mathscr D}}{\ell^{d,\infty}(\mathscr D)} 
  \lesssim \Norm{[b,T]}{S^{d,\infty}(L^2(\mu))}.
\end{equation}
The finiteness of each $\operatorname{osc}_\nu(f,B_Q)$ shows in particular that $b\in L^1_{\loc}(\nu)$. This gives us access to Proposition \ref{prop:cont-discrete-osc} and part \eqref{it:M1p-weakLp} of Theorem \ref{thm:HK-main}, and these show that
\begin{equation*}
\begin{split}
  \Norm{b}{\dot M^{1,d}(\nu)}
  &\sim \Norm{\operatorname{osc}_\nu(f,B(\cdot,\cdot))}{L^{d,\infty}(\nu_d)}
  \qquad\text{by \eqref{eq:HK-bound}} \\
  &\sim \Norm{\{\operatorname{osc}_\nu(f,B_Q)\}_{Q\in\mathscr D}}{\ell^{d,\infty}(\mathscr D)} 
  \qquad\text{by \eqref{eq:disc-cont-bound}} \\
  &  \lesssim \Norm{[b,T]}{S^{d,\infty}(L^2(\mu))}
  \qquad\text{by \eqref{eq:by-osc<comm}},
\end{split}
\end{equation*}
which proves \eqref{eq:Sob<comm-bound}.

In particular, $b\in\dot M^{1,d}(\nu)$. Under this condition, we also have access to part \eqref{it:HK-asymp} of Theorem \ref{thm:HK-main}. Thus
\begin{equation*}
\begin{split}
  \Norm{b}{\dot M^{1,d}(\nu)}
  &\sim \liminf_{\lambda\to 0}\lambda\cdot\nu_d(\{\operatorname{osc}_{\nu}(f,B(\cdot,\cdot))>\lambda\})^{\frac 1d}
  \qquad\text{by \eqref{eq:HK-asymp}} \\
  &\sim \liminf_{\lambda\to 0}\lambda\cdot(\#\{Q\in\mathscr D:
     \operatorname{osc}_{\nu}(f,B_Q)>\lambda\})^{\frac 1d}
   \qquad\text{by \eqref{eq:disc-cont-asymp}} \\
  &\lesssim
  \Big(\liminf_{n\to\infty} a_n([b,T])\cdot n^{\frac1p}\Big)^\theta
  \Norm{[b,T]}{S^{p,\infty}(L^2(\mu))}^{1-\theta}
  \qquad\text{by \eqref{eq:osc<comm-asymp}},
\end{split}
\end{equation*}
for any $\theta\in(0,1-\frac1d)$ in the last step.
\end{proof}

The following corollary is a restatement of case \eqref{it:main-reg} of Theorem \ref{thm:main}; we now complete its proof.

\begin{corollary}\label{cor:Sob=comm}
Let $\mu,\nu$ be two doubling measures on a metric space $(X,\rho)$ that satisfy the $A_\infty$ relation \eqref{eq:Ainfty}.
Let $(X,\rho,\nu)$ be Ahlfors $d$-regular for some $d\in(1,\infty)$ and satisfy the $q$-Poincar\'e inequality for some $q\in[1,d)$.
Let $b\in L^1_{\loc}(\mu)$ and $T$ be a non-degenerate H\"older $\eta$-regular $\mu$-singular integral with $\eta>(1-\frac{d}{2})_+$.
Then $[b,T]\in S^{d,\infty}(L^2(\mu))$ if and only if $b\in\dot M^{1,d}(\nu)$, and in this case
\begin{subequations}\label{eq:new-regular}
\begin{align}
    \Norm{ b }{ \dot M^{1,d}(\nu)}
    &\sim\Norm{[b,T]}{S^{d,\infty}(L^2(\mu))},
    \label{eq:new-bound} \\
    &\sim\liminf_{n\to\infty}a_n([b,T])\cdot n^{\frac1d}
      \sim\limsup_{n\to\infty}a_n([b,T])\cdot n^{\frac1d}.
    \label{eq:new-asymp}
\end{align}
\end{subequations}
\end{corollary}

\begin{proof}
The norm equivalence \eqref{eq:new-bound} is a restatement of \cite[Theorem 1.4(3)]{Hyt:Osc}. In fact, we already had the lower bound in \eqref{eq:Sob<comm-bound}, so we would only need to quote the upper bound
\begin{equation}\label{eq:reg-upper-bound}
  \Norm{[b,T]}{S^{d,\infty}(L^2(\mu))}\lesssim\Norm{ b }{ \dot M^{1,d}(\nu)},
\end{equation}
which is where the H\"older $\eta$-regularity assumption arises, due to its presence in \cite[Theorem 1.4]{Hyt:Osc}. Suppose then that that the two equivalent sides of \eqref{eq:new-bound} are finite. Then
\begin{equation*}
\begin{split}
   \Norm{b}{\dot M^{1,d}(\nu)}
   &\lesssim
   \Big(\liminf_{n\to\infty} a_n([b,T])\cdot n^{\frac1p}\Big)^\theta
  \Norm{[b,T]}{S^{p,\infty}(L^2(\mu))}^{1-\theta}\qquad\text{by \eqref{eq:Sob<comm-asymp}} \\
   &\lesssim
    \Big(\liminf_{n\to\infty} a_n([b,T])\cdot n^{\frac1d}\Big)^\theta
  \Norm{b}{\dot M^{1,p}(\nu)}^{1-\theta}\qquad\text{by \eqref{eq:reg-upper-bound}}.
\end{split}
\end{equation*}
If $\Norm{b}{\dot M^{1,d}(\nu)}>0$, this immediately gives
\begin{equation}\label{eq:Sob<liminf}
  \Norm{b}{\dot M^{1,d}(\nu)}
  \lesssim\liminf_{n\to\infty} a_n([b,T])\cdot n^{\frac1d},
\end{equation}
and the same conclusion is trivial if $\Norm{b}{\dot M^{1,d}(\nu)}=0$. Combining everything, we have
\begin{equation*}
\begin{split}
    \Norm{b}{\dot M^{1,d}(\nu)}
  &\lesssim\liminf_{n\to\infty} a_n([b,T])\cdot n^{\frac1d}\qquad\text{by \eqref{eq:Sob<liminf}} \\
  &\leq\limsup_{n\to\infty} a_n([b,T])\cdot n^{\frac1d}
  \leq\Norm{[b,T]}{S^{d,\infty}(L^2(\mu))}\qquad\text{trivially} \\
  &\phantom{\leq\limsup_{n\to\infty} a_n([b,T])\cdot n^{\frac1d}}\
  \lesssim\Norm{b}{\dot M^{1,d}(\nu)}\qquad\text{by \eqref{eq:reg-upper-bound}}.
\end{split}
\end{equation*}
Thus all these quantities are equivalent, completing the proof of \eqref{eq:new-regular}.
\end{proof}

\section{The commutator upper bound revisited}\label{sec:upper}

In this section, we provide the proof of case \eqref{it:main-new} of Theorem \ref{thm:main}, but let us first discuss some background for this.

In the proof of case \eqref{it:main-reg} of Theorem \ref{thm:main}, restated as Corollary \ref{cor:Sob=comm}, a key ingredient borrowed from elsewhere was the commutator upper bound \eqref{eq:reg-upper-bound}. We quoted this from \cite[Theorem 1.4]{Hyt:Osc}, which in turn builds on \cite[Theorem 1.1]{Hyt:Sp}. For more specific situations in place of the abstract spaces $X$ and operators $T$ considered in \cite{Hyt:Osc,Hyt:Sp}, several different proofs of \eqref{eq:reg-upper-bound} are available, none of them very easy.

For highly regular (namely, $\abs{\partial_x^\alpha\partial_y^\beta K(x,y)}\lesssim\abs{x-y}^{-d-\abs{\alpha}-\abs{\beta}}$ for all multi-indices $\alpha,\beta$ up to a high degree) singular integrals $T$ on $\R^d$, it was shown in \cite{RS:end,RS:NWO} that
\begin{equation}\label{eq:RS-upper}
  \Norm{[b,T]}{S^{d,\infty}(L^2(\R^d))}
  \lesssim\Norm{\{\operatorname{osc}(b,B_Q)\}_{Q\in\R^d}}{\ell^{d,\infty}(\mathscr D)}
  \lesssim\Norm{b}{\dot W^{1,d}(\R^d)};
\end{equation}
the first estimate in \eqref{eq:RS-upper} is contained in \cite[Corollary 2.9]{RS:NWO} and the second one in \cite[Theorem 2.2]{RS:end}. The proof of the first step of \eqref{eq:RS-upper} in \cite{RS:NWO} uses local Fourier series expansions of the highly regular kernels to decompose them in terms of so-called {\em nearly weakly orthogonal} (NWO) sequences---the key tool of \cite{RS:NWO}. These considerations have been extended to  highly regular singular kernels on Carnot groups in \cite{LXY}, using so-called Alpert bases \cite{Alpert} as a substitute for Fourier series. Relaxing the high regularity assumptions of both \cite{LXY,RS:NWO}, an extension of \eqref{eq:RS-upper} for standard H\"older $\eta$-regular singular integrals on $\R^d$ was found in \cite[Theorem 1.4]{WZ:median}, replacing the NWO decomposition by the dyadic representation formula of \cite{Hyt:A2}. 
An extension of this representation to abstract metric measure spaces is also behind the results of \cite{Hyt:Osc,Hyt:Sp} that we quoted for \eqref{eq:reg-upper-bound}.

For the important special case of the classical Riesz transforms on $\R^d$ and their analogues on the Heisenberg group and the Bessel setting, completely different proofs of \eqref{eq:reg-upper-bound}---by-passing the intermediate step in \eqref{eq:RS-upper} and replacing the related NWO/dyadic decompositions by various methods of noncommutative analysis---are due to \cite{LMSZ} and \cite{FLMSZ,FLSZ}, respectively.

Our goal in this section is to show that significant cases of \eqref{eq:RS-upper} may also be obtained by an arguable more elementary argument, building on the techniques of Janson--Wolff \cite{JW:82} that predate Rochberg--Semmes \cite{RS:NWO} by several years. While the method of \cite{JW:82} has been adapted for commutator Schatten class estimates with non-critical parameters in a variety of settings (cf.\ \cite[Section 4.2]{FLL:23}, \cite[Section 2.2]{LXY}, \cite[Section 3.2]{FLLVW:Neu}), the possibility of extending this approach even to the critical case seems to have been overlooked before.

Nevertheless,  a new approach to the commutator upper bounds at the critical index can readily be executed as follows. We denote by
\begin{equation*}
  \mathcal I_K f(x):=\int_X K(x,y)f(y)\ud\mu(y)
\end{equation*}
the $\mu$-integral operator with kernel $K$, and by
\begin{equation*}
  V(x,y):=\mu(B(x,\rho(x,y)))
\end{equation*}
the $\mu$-measure of the ball at centre $x$ and radius $\rho(x,y)$.

\begin{proposition}
Let $(X,\rho,\mu)$ be a doubling metric measure space and $p\in(2,\infty)$. Then
\begin{subequations}
\begin{align}
  \Norm{\mathcal I_{V^{-1}F}}{ S^p(L^2(\mu))}  &\lesssim \Norm{F}{L^p(V^{-2})}, 
  \label{eq:JW-strong} \\
  \Norm{\mathcal I_{V^{-1}F}}{ S^{p,\infty}(L^2(\mu))}  &\lesssim \Norm{F}{L^{p,\infty}(V^{-2})}.
  \label{eq:JW-weak}
\end{align}
\end{subequations}
\end{proposition}

\begin{proof}
Estimate \eqref{eq:JW-strong} is \cite[Proposition 5.8]{Hyt:Sp}, elaborating on earlier work of \cite{JW:82}. Estimate \eqref{eq:JW-weak} follows by real interpolation: if $\frac{1}{p}=\frac{1-\theta}{p_0}+\frac{\theta}{p_1}$, then $(S^{p_0},S^{p_1})_{p,\infty}=S^{p,\infty}$ \cite[Th\'eor\`eme 2]{Merucci} and $(L^{p_0},L^{p_1})_{p,\infty}=L^{p,\infty}$ \cite[Theorem 5.2.1]{BL:book}.
\end{proof}

\begin{proposition}\label{prop:new-upper}
Let $d>2$, and let $(X,\rho,\mu)$ be  Ahflors $d$-regular. If $T$ is a $\mu$-singular integral operator, then
\begin{equation*}
   \Norm{[b,T]}{ S^{d,\infty}(L^2(\mu))}  \lesssim \Norm{b}{\dot M^{1,d}(\mu)}.
\end{equation*}
\end{proposition}

\begin{proof}
If $T$ has kernel $K$, then $[b,T]$ has kernel $K\Delta b$, where $\Delta b(x,y):=b(x)-b(y)$. Hence
\begin{equation*}
\begin{split}
   \Norm{[b,T]}{ S^{d,\infty}(L^2(\mu))}
  &=\Norm{\mathcal I_{KB}}{ S^{d,\infty}(L^2(\mu))}\qquad\text{denoting } B:=\Delta b \\
  &=\Norm{\mathcal I_{V^{-1}VKB}}{ S^{d,\infty}(L^2(\mu))} \\
  &\lesssim \Norm{VKB}{L^{d,\infty}(V^{-2})}
  \qquad\text{by \eqref{eq:JW-weak} with }F=VKB \\
  &\lesssim\Norm{B}{L^{d,\infty}(V^{-2})}\qquad\text{since }\abs{K}\lesssim V^{-1} \\
\end{split}
\end{equation*}

By definition, if $h$ is a Haj\l{}asz upper gradient of $b$, then
\begin{equation*}
  \abs{B(x,y)}=\abs{b(x)-b(y)} \leq\rho(x,y) (h(x)+h(y)),
\end{equation*}
and the two terms are symmetric. Thus
\begin{equation*}
\begin{split}
  \iint_{\abs{B(x,y)}>\lambda}\frac{\ud\mu(x)\ud\mu(y)}{V(x,y)^2}
   &\lesssim  \iint_{\rho(x,y)h(x)>\lambda}\frac{\ud\mu(x)\ud\mu(y)}{V(x,y)^2} \\
   &=\int_X\Big(\int_{\rho(x,y)>\lambda/h(x)}\frac{\ud\mu(y)}{V(x,y)^2}\Big)\ud\mu(x).
\end{split}
\end{equation*}
Using Ahlfors $d$-regularity, we have
\begin{equation*}
\begin{split}
  \int_{\rho(x,y)>r}\frac{\ud\mu(y)}{V(x,y)^2}
  \lesssim\sum_{k=0}^\infty\int_{2^kr <\rho(x,y)\leq 2^{k+1}r}\frac{\ud\mu(y)}{(2^k r)^{2d}}
  \lesssim\sum_{k=0}^\infty\frac{1}{(2^k r)^d}\sim r^{-d}.
\end{split}
\end{equation*}
Substituting the result with $r=\lambda/h(x)$ into the earlier display, we obtain
\begin{equation*}
  \int_X\Big(\int_{\rho(x,y)>\lambda/h(x)}\frac{\ud\mu(y)}{V(x,y)^2}\Big)\ud\mu(x)
  \lesssim\int_X\Big(\frac{h(x)}{\lambda}\Big)^d\ud\mu(x)
  =\lambda^{-d}\Norm{h}{L^d(\mu)}^d.
\end{equation*}
Thus
\begin{equation*}
  \Norm{[b,T]}{S^{d,\infty}(L^2(\mu))}
  \lesssim\Norm{B}{L^{d,\infty}(V^{-2})}
  \lesssim\Norm{h}{L^d(\mu)},
\end{equation*}
and taking the infimum over Haj\l{}asz upper gradients $h$ of $b$, we obtain the claimed estimate.
\end{proof}

The following corollary is a restatement of case \eqref{it:main-new} of Theorem \ref{thm:main}; we now complete its proof.

\begin{corollary}\label{cor:main-new}
Let $d>2$, and let $(X,\rho,\mu)$ be an Ahflors $d$-regular with the $q$-Poincar\'e inequality for some $q\in[1,d)$. If $T$ is a strongly non-degenerate $\mu$-singular integral operator and $b\in L^1_{\loc}(\mu)$, then $[b,T]\in S^{d,\infty}(L^2(\mu))$ if and only if $b\in\dot M^{1,d}(\mu)$, and in this case
\begin{equation*}
\begin{split}
  \Norm{b}{\dot M^{1,d}(\mu)}
  &\sim   \Norm{[b,T]}{ S^{d,\infty}(L^2(\mu))}  \\
   &\sim\liminf_{n\to\infty}a_n([b,T])\cdot n^{\frac1d}
   \sim\limsup_{n\to\infty}a_n([b,T])\cdot n^{\frac1d}.
\end{split}
\end{equation*}
\end{corollary}

\begin{proof}
One can repeat the proof of Corollary \ref{cor:Sob=comm} , making only the following change. In the said proof, the estimate \eqref{eq:reg-upper-bound} was obtained by appealing to the assumed H\"older $\eta$-regularity and \cite[Theorem 1.4]{Hyt:Osc}. Now, under the present assumptions, the same estimate \eqref{eq:reg-upper-bound} (with $\mu=\nu$) is obtained from Proposition \ref{prop:new-upper}. The rest of the proof remains unchanged, except for taking $\nu=\mu$ throughout.
\end{proof}

\subsection*{AI declaration}

The author consulted GPT-5.6 Luna for general background about traces (relevant to Corollary \ref{cor:trace}), but the information was subsequently checked from other sources as cited in the text. Other than this, there was no AI involvement in the rest of the work. The author takes full responsibility of the entire content.


\end{document}